\documentclass[12pt]{article}
\usepackage{amsmath,amssymb,amsthm}
\usepackage{comment}
\numberwithin{equation}{section}
\newtheorem{thm}{Theorem}[section]
\newtheorem{prop}[thm]{Proposition}
\newtheorem{cor}[thm]{Corollary}
\newtheorem{exam}[thm]{Example}
\newtheorem{rem}[thm]{Remark}
\newtheorem{lem}[thm]{Lemma}

\newtheorem{assum}[thm]{Assumption}

\def\Capa{\mathop{\rm Cap}}

\makeatletter
\begin{document}

\title{Transience of symmetric non-local Dirichlet forms}
\author{Yuichi Shiozawa\thanks{Department of Mathematics,
Graduate School of Science,
Osaka University, Toyonaka, Osaka, 560-0043,
Japan; \texttt{shiozawa@math.sci.osaka-u.ac.jp}}
\thanks{Supported in part by JSPS KAKENHI No.\ JP17K05299.}}
\maketitle

\begin{abstract}
We establish transience criteria for symmetric non-local Dirichlet forms on $L^2({\mathbb R}^d)$ 
in terms of the coefficient growth rates at infinity. 
Applying these criteria, we find a necessary and sufficient condition 
for recurrence of Dirichlet forms of symmetric stable-like 
with unbounded/degenerate coefficients. 
This condition indicates that 
both of the coefficient growth rates of small and big jump parts 
affect the sample path properties of the associated symmetric jump processes. 
\end{abstract}

\section{Introduction}

In this paper, we establish transience criteria 
for a symmetric non-local Dirichlet form on $L^2({\mathbb R}^d)$ 
in terms of the growth rates of the coefficients at infinity. 
Applying these criteria, 
we obtain a necessary and sufficient condition 
for recurrence of Dirichlet forms of symmetric stable-like 
with unbounded/degenerate coefficients. 

In the theory of stochastic processes, 
it is one of the important research subjects 
to study the global path properties of a symmetric Markov process. 
These properties are closely related to 
the analytic and geometric conditions 
such as the coefficient growth of the generator and 
the volume growth of the state space 
(see, e.g., \cite{G99,G21,P95,S17} and references therein for details).

Here we focus on recurrence and transience of symmetric Markov processes 
generated by Dirichlet forms. 
Roughly speaking, we mean by recurrence the nature of particles  
hitting every open set infinitely many times almost surely. 
On the other hand,  
we mean by transience the nature of particles escaping to infinity eventually almost surely. 
There are many recurrence criteria 
for symmetric Markov processes 
in terms of the analytic and geometric conditions as mentioned before
(see, for instance, \cite{G99,G21,I78,MUW12,OU20+,O95,O96,OU15,S94}); 
however, there are less such transience criterion 
in spite of their importance. 

For the Brownian motion on a spherically symmetric Riemannian manifold, 
we know a necessary and sufficient condition for recurrence 
in terms of the surface area of a ball 
(see, e.g., \cite[Theorem 5.1 in Section 9]{P95} or \cite[Proposition 3.1]{G99}). 
Ichihara \cite{I78} established sharp recurrence and transience criteria 
for the symmetric diffusion process on ${\mathbb R}^d$  
generated by a strongly local Dirichlet form with unbounded/degenerate coefficient 
(see also \cite[Section 6]{P95} for the exposition on this topic, 
and \cite{I82} and \cite[Section 9]{P95} for the Brownian motion on a Riemannian manifold).

As for symmetric jump processes on ${\mathbb R}^d$, 
\^Okura-Uemura \cite{OU15} and Okamura-Uemura \cite{OU20+} 
obtained a sufficient condition for recurrence 
in terms of the coefficient growth rates of the associated non-local Dirichlet forms. 
This condition is consistent with that of \cite{I78}. 
However, as far as the author knows, 
we have no corresponding condition for transience. 
Our objective in this paper is to establish transience criteria 
for a class of symmetric jump processes on ${\mathbb R}^d$ (Theorems \ref{thm:tran-sm} and \ref{thm:tran-bg}). 
Using these criteria, we can verify the sharpness of the results of \cite{OU15,OU20+} 
in terms of the polynomial growth rates of the coefficients 
(Corollary \ref{cor:tran-iff}). 

We note here that, in contrast with local Dirichlet forms, 
non-local ones have two kinds of the coefficients 
corresponding to the small and big jump parts.  
On account of this, we are interested in clarifying the relation between 
the growth rates of these coefficients and the global properties of sample paths. 
From this point of view, we established conservativeness criteria for 
symmetric jump processes (\cite{S15,SU14}). 
We also find in \cite{SW20+} a necessary and sufficient condition 
for the associated Markovian semigroups being $L^2$-compact. 
These two results say that 
the coefficient growth rate of the big jump is irrelevant to 
the validity of the conservativeness and $L^2$-compactness. 
Our results in this paper prove that 
both of the coefficients do affect the sample path properties.

To establish transience criteria for a non-local Dirichlet form, 
we adopt the so-called Lyapunov method; 
that is, we apply some test function to the generator and 
use the martingale characterization of the generator. 
This method is standard for diffusion processes (see, e.g., \cite{P95} and references therein) 
and applicable to jump processes (see, e.g., \cite[Lemma 4]{KST02}). 
Here we use this method separately to the generators 
corresponding to small and big jump parts. 
For the small jump part,  
we develop calculation in \cite[Proposition 2.5]{SW20+} 
to find the influence of the coefficient growth rate and spatial dimension on transience (Proposition \ref{prop:liyap}). 
Concerning the big jump part, 
we need elementary but involved calculation 
in order to ensure the existence of a test function  (Lemma \ref{lem:derivative}). 
Because of this, we impose a strong restriction on the big jump part 
(see Subsection \ref{subsect:big}).
Our approach for the big jump part is motivated by \cite[Lemmas 4 and 7]{KST02}. 

The rest of this paper is organized as follows: 
In Section \ref{sect:pre}, we summarize some materials on 
the Dirichlet form theory relevant to recurrence and transience.
In Section \ref{sect:tran}, we present transience criteria for symmetric jump processes 
and carry out the Lyapunov method. 
Some of the involved calculation are postponed to  Section \ref{sect:lem}. 
Appendix contains a version of the recurrence criterion for non-local Dirichlet forms 
established by \^Okura \cite{O95} and \^Okura-Uemura \cite{OU15}, 
and an application of this criterion. 
This appendix makes the proof of Corollary \ref{cor:tran-iff} self-contained.
 
\section{Preliminaries}\label{sect:pre}
We collect some materials 
related to recurrence and transience of Dirichlet forms from \cite{FOT11}. 

\subsection{Analytic part}
Let $(X,{\cal B},m)$ be a $\sigma$-finite measure space and 
$(u,v)=\int_X uv\,{\rm d}m$ for $u,v\in L^2(X;m)$.
Let $\{T_t\}_{t>0}$ be a strongly continuous Markovian semigroup on $L^2(X;m)$ 
and $({\cal E},{\cal F})$ an associated  Dirichlet form. 
Namely, $({\cal E},{\cal F})$ is a closed and Markovian symmetric form on $L^2(X;m)$ given by 
$${\cal F}=\left\{u\in L^2(X;m) \mid \lim_{t\rightarrow\infty}\frac{1}{t}(u-T_tu,u)<\infty \right\}, \ 
{\cal E}(u,u)=\lim_{t\rightarrow\infty}\frac{1}{t}(u-T_tu,u) \ \text{for $u\in {\cal F}$}$$
(see \cite[Lemma 1.3.4]{FOT11}).

Let 
$$S_t f=\int_0^t T_sf\,{\rm d}s, \quad f\in L^2(X;m).$$ 
Here the integral is defined as the strong convergence limit in $L^2(X;m)$ 
of the Riemann sum. 
We can then extend $S_t$ and $T_t$ from $L^1(X;m)\cap L^2(X;m)$ 
to $L^1(X;m)$ uniquely. We use the same notations for such extensions. 
Let 
$$L_+^1(X;m)=\left\{f\in L^1(X;m) \mid f\geq 0, \ \text{$m$-a.e.}\right\}$$
and 
$$Gf=\lim_{N\rightarrow\infty}S_N f, \quad f\in L_+^1(X;m).$$
We say that $\{T_t\}_{t>0}$ (or $({\cal E},{\cal F})$) is {\it recurrent} 
if for any $f\in L_+^1(X;m)$, 
either $Gf=0$ $m$-a.e.\ or $Gf=\infty$ $m$-a.e.\ holds. 
We also say that $\{T_t\}_{t>0}$ (or $({\cal E},{\cal F})$) is  {\it transient} if 
$Gf<\infty$ $m$-a.e.\ for any $f\in L_+^1(X;m)$. 

The next theorem is a comparison principle 
for recurrence and transience of Dirichlet forms. 
\begin{thm}\label{thm:comparison} {\rm (\cite[Theorem 1.6.4]{FOT11})}
Let $({\cal E}^{(1)},{\cal F}^{(1)})$ and  $({\cal E}^{(2)},{\cal F}^{(2)})$ be 
Dirichlet forms on $L^2(X;m)$ such that 
${\cal F}^{(2)}\subset {\cal F}^{(1)}$. 
If there exists $C>0$ such that  
$${\cal E}^{(1)}(u,u)\leq C{\cal E}^{(2)}(u,u) \quad \text{for any $u\in {\cal F}^{(2)}$,}$$
then the following assertions hold{\rm :}
\begin{enumerate}
\item[{\rm (1)}] if $({\cal E}^{(2)}, {\cal F}^{(2)})$ is recurrent, then so is  $({\cal E}^{(1)}, {\cal F}^{(1)})${\rm ;}
\item[{\rm (2)}] if $({\cal E}^{(1)}, {\cal F}^{(1)})$ is transient, then so is  $({\cal E}^{(2)}, {\cal F}^{(2)})$.
\end{enumerate}
\end{thm}

A measurable set $A\subset X$ is called ($T_t$-){\it invariant} if 
for any $f\in L^2(X;m)$ and $t>0$, 
$$T_t({\bf 1}_Af)={\bf 1}_AT_tf, \quad \text{$m$-a.e.}$$
We say that $\{T_t\}_{t>0}$ (or $({\cal E},{\cal F})$) is {\it irreducible} 
if any invariant set $A\subset X$ satisfies $m(A)=0$ or $m(X\setminus A)=0$.
We then know

\begin{lem}\label{lem:irreducible} {\rm (\cite[Lemma 1.6.4 (iii)]{FOT11})}
If $({\cal E},{\cal F})$ is irreducible, then it is recurrent or transient. 
\end{lem}

In what follows, let $X$ be a locally compact separable metric space 
and $m$ a positive Radon measure on $X$ with full support. 
Let $X_{\Delta}=X\cup \{\Delta\}$ be a one-point compactification of $X$.  
We write ${\cal O}$ for the totality of open sets in $X$.

Let ${\cal E}_1(u,u)={\cal E}(u,u)+\|u\|_{L^2(X;m)}^2$ for $u\in {\cal F}$. 
For $A\in {\cal O}$, let  
$${\cal L}_A=\left\{u\in {\cal F} \mid u\geq 1 \ \text{$m$-a.e.\ on $A$}\right\}$$
and 
$$
\Capa(A)=
\begin{cases}
\inf\{{\cal E}_1(u,u) \mid u\in {\cal L}_A\} & \text{if ${\cal L}_A\ne \emptyset$,}\\
\infty & \text{if ${\cal L}_A=\emptyset$.}
\end{cases}$$
We define the {\it capacity} of a set $B\subset X$ by 
$$\Capa(B)=\inf_{A\in {\cal O}, A\supset B}\Capa(A).$$
Note that if $B$ is measurable with $m(B)>0$, then $\Capa(B)>0$ by definition. 

For $A\subset X$, a statement depending on $x\in A$ is said 
to hold quasi everywhere (q.e.\ for short) on $A$ 
if there exists $N\subset X$ of zero capacity 
such that the statement holds for every $x\in A\setminus N$.

\subsection{Probabilistic part}
Let ${\mathbf M}=(\Omega,{\cal F}, \{X_t\}_{t\geq 0},\{P_x\}_{x\in X},\{\theta_t\}_{t\geq 0}, \zeta)$ 
be a $m$-symmetric Borel right process on $X$. 
Here $\zeta=\inf\left\{t>0 \mid X_t\in \Delta\right\}$ is the lifetime of ${\mathbf M}$ 
and $\theta_t:\Omega\to \Omega$ is the shift of paths defined 
by the relation $X_s\circ\theta_t=X_{s+t}$ for every $s\geq 0$.

We write ${\cal B}(X)$ for the set of all Borel measurable subsets of $X$ 
and 
$${\cal B}(X_{\Delta})={\cal B}(X)\cup \left\{B\cup\{\Delta\} \mid B\in {\cal B}(X)\right\}.$$
A set $B\subset X$ is called {\it nearly Borel measurable}  
if there exist $B_1,B_2\in {\cal B}(X_{\Delta})$ such that  
$B_1\subset B\subset B_2$ and for any probability measure $\mu$ on $X_{\Delta}$, 
$$P_x(\text{$X_t\in B_2\setminus B_1$ for some $t\geq 0$})=0, \quad \text{$\mu$-a.e.\ $x\in X$}.$$
For a nearly Borel set $A$, 
let $\sigma_A=\inf\{t>0 \mid X_t\in A\}$ denote the hitting time of ${\mathbf M}$ to $A$. 
A set $N\subset X$ is called {\it exceptional} if there exists a nearly Borel set $\tilde{N}$ 
such that $\tilde{N}\supset N$ and $P_x(\sigma_{\tilde{N}}<\infty)=0$ for $m$-a.e.\ $x\in X$.  
We say that a set $N\subset X$ is {\it properly exceptional} 
if it is nearly Borel such that $m(N)=0$ and 
$X\setminus N$ is ${\mathbf M}$-invariant, that is, 
$$P_x(\text{$X_t\in (X\setminus N)_{\Delta}$ and $X_{t-}\in (X\setminus N)_{\Delta}$ for any $t>0$})=1, 
\quad x\in X\setminus N.$$ 
Here $(X\setminus N)_{\Delta}=(X\setminus N)\cup \{\Delta\}$ and $X_{t-}=\lim_{s\rightarrow t-0}X_s$.
Note that any properly exceptional set is exceptional by definition.

Let $\{p_t\}_{t>0}$ be an $m$-symmetric Markovian transition function of ${\mathbf M}$ 
on $(X,{\cal B}(X))$. 
Then it uniquely associates an strongly continuous Markovian semigroup $\{T_t\}_{t>0}$ 
and thus a Dirichlet form $({\cal E},{\cal F})$ on $L^2(X;m)$ (\cite[\S 4.2]{FOT11}). 
The next theorem shows a probabilistic consequence of the recurrence of irreducible Dirichlet forms. 
\begin{thm}\label{thm:irr-rec}{\rm (\cite[Theorem 4.7.1 (iii)]{FOT11})} 
If $({\cal E},{\cal F})$ is irreducible recurrent, 
then for any nearly Borel non-exceptional set $B$, 
$$P_x(\sigma_B\circ \theta_n<\infty \ \text{for any $n\geq 0$})=1, \quad \text{q.e.\ $x\in X$}.$$ 
\end{thm}

Lemma \ref{lem:irreducible} and Theorem \ref{thm:irr-rec} imply that  
an irreducible Dirichlet form $({\cal E},{\cal F})$ is transient if there exists a nearly Borel 
non-exceptional set $B$ such that for any exceptional set $N\subset X$,
$P_x(\sigma_B<\infty)<1$ for some $x\in X\setminus N$.

Let $C_0(X)$ denote the totality of continuous functions on $X$ with compact support. 
A Dirichlet form $({\cal E},{\cal F})$ on $L^2(X;m)$ is called {\it regular}  
if  ${\cal F}\cap C_0(X)$ is dense both in ${\cal F}$ 
with respect to the norm 
\begin{equation}\label{eq:norm}
\|u\|_{{\cal E}_1}:=\sqrt{{\cal E}_1(u,u)}, \quad u\in {\cal F},
\end{equation}
and in $C_0(X)$ with respect to the uniform norm. 

Suppose that the Dirichlet form $({\cal E},{\cal F})$ is regular. 
Then there exists an $m$-symmetric Hunt process ${\mathbf M}$ on $X$ 
whose Dirichlet form is $({\cal E},{\cal F})$ (\cite[Theorem 7.2.1]{FOT11}).  
We also know that a set $N\subset X$ is exceptional 
if and only if $\Capa(N)=0$ (\cite[Theorem 4.2.1]{FOT11}).  

\section{Transience}\label{sect:tran}
In this section, we provide a sufficient condition 
for transience of a non-local Dirichlet form on $L^2({\mathbb R}^d)$ 
with unbounded coefficients. 
To do so, we separate the small and big jump parts of this Dirichlet form.

\subsection{Setting}\label{sub:set}
Let $J(x,{\rm d}y)$ be a nonnegative kernel on 
$({\mathbb R}^d,{\cal B}({\mathbb R}^d))$.
\begin{assum}\label{assum:regular}
The measure $J({\rm d}x\, {\rm d}y)=J(x,{\rm d}y)\,{\rm d}x$ is symmetric on ${\mathbb R}^d\times {\mathbb R}^d$ and 
the function 
$$x\mapsto \int_{{\mathbb R}^d}(1\wedge |x-y|^2)J(x,{\rm d}y)$$
belong to $L_{{\rm loc}}^1({\mathbb R}^d)$. 
\end{assum}

Let $({\cal E},{\cal D}({\cal E}))$ be a quadratic form on $L^2({\mathbb R}^d)$ defined by 
\begin{equation}\label{eq:form}
\begin{split}
{\cal D}({\cal E})
&=
\left\{u\in L^2({\mathbb R}^d) :
\iint_{{\mathbb R}^d\times{\mathbb R}^d}(u(x)-u(y))^2\,J(x, {\rm d}y){\rm d}x
<\infty\right\}, \\
{\cal E}(u,v)
&=\iint_{{\mathbb R}^d\times{\mathbb R}^d}(u(x)-u(y))(v(x)-v(y))\,J(x, {\rm d}y){\rm d}x,
\quad u,v\in {\cal D}({\cal E}).
\end{split}
\end{equation}
Let $\|\cdot\|_{{\cal E}_1}$ be the norm on ${\cal D}({\cal E})$ defined similarly to \eqref{eq:norm},  
and let $C_0^{\infty}({\mathbb R}^d)$ be the set of smooth functions on ${\mathbb R}^d$ 
with compact support. 
Then under Assumption \ref{assum:regular}, 
$C_0^{\infty}({\mathbb R}^d)\subset {\cal D}({\cal E})$ and 
$({\cal E},C_0^{\infty}({\mathbb R}^d))$ is closable with respect to the norm $\|\cdot \|_{{\cal E}_1}$.  
Hence if $({\cal E},{\cal F})$ denotes 
the $\|\cdot \|_{{\cal E}_1}$-closure of $({\cal E},C_0^{\infty}({\mathbb R}^d))$, 
then it is a regular Dirichlet form on $L^2({\mathbb R}^d)$.

Let
$$J^{(1)}({\rm d}x\,{\rm d}y)={\bf 1}_{\{|x-y|<1\}}\,J({\rm d}x\,{\rm d}y), \quad 
J^{(2)}({\rm d}x\,{\rm d}y)={\bf 1}_{\{|x-y|\geq 1\}}\,J({\rm d}x\,{\rm d}y).$$
For $i=1,2$, let $({\cal E}^{(i)},{\cal D}({\cal E}^{(i)}))$ be a quadratic form on $L^2({\mathbb R}^d)$ 
defined by 
\begin{equation*}
\begin{split}
{\cal D}({\cal E}^{(i)})&=
\left\{u\in L^2({\mathbb R}^d) :
\iint_{{\mathbb R}^d\times{\mathbb R}^d}(u(x)-u(y))^2\,J^{(i)}({\rm d}x\,{\rm d}y)
<\infty\right\}, \\
{\cal E}^{(i)}(u,v)
&=\iint_{{\mathbb R}^d\times{\mathbb R}^d}(u(x)-u(y))(v(x)-v(y))\,J^{(i)}({\rm d}x\,{\rm d}y),
\quad u,v\in {\cal D}({\cal E}^{(i)}).
\end{split}
\end{equation*}
Let $\|\cdot \|_{{\cal E}_1^{(i)}}$ be the norm on ${\cal D}({\cal E}^{(i)})$ 
defined similarly to \eqref{eq:norm}. 
If $({\cal E}^{(i)},{\cal F}^{(i)})$ is the $\|\cdot \|_{{\cal E}_1^{(i)}}$-closure of 
$({\cal E}^{(i)}, C_0^{\infty}({\mathbb R}^d))$,
then it is a regular Dirichlet form on $L^2({\mathbb R}^d)$. 
By definition, we have for $i=1,2$,  
\begin{equation}\label{eq:comp}
{\cal F}\subset {\cal F}^{(i)} \quad \text{and} \quad {\cal E}^{(i)}(u,u)\leq {\cal E}(u,u) \quad \text{for any $u\in {\cal F}$.}  
\end{equation}

Let $a_i(x) \ (i=1,2)$ be positive Borel measurable functions on ${\mathbb R}^d$ 
and 
\begin{equation*}
c(x,y)=(a_1(x)+a_1(y)){\bf 1}_{\{|x-y|<1\}}
+(a_2(x)+a_2(y)){\bf 1}_{\{|x-y|\geq 1\}}.
\end{equation*}
We further make the next assumption on the jumping measure $J({\rm d}x\,{\rm d}y)$.

\begin{assum}\label{assum:lower}
There exists a kernel $J_0(x,{\rm d}y)$ on ${\mathbb R}^d\times {\cal B}({\mathbb R}^d)$ 
satisfying the following{\rm :}
\begin{enumerate}
\item[{\rm (i)}] The measure $J_0({\rm d}x \, {\rm d}y)=J_0(x,{\rm d}y){\rm d}x$ is symmetric. 
Moreover, there exists a positive and rotationally invariant 
Borel measure $\nu({\rm d}z)$ on ${\mathbb R}^d$ such that 
$\int_{{\mathbb R}^d\setminus\{0\}}(1\wedge |z|^2)\,\nu({\rm d}z)<\infty$  and  
for any $x\in {\mathbb R}^d$ and $A\in {\cal B}({\mathbb R}^d)$,
\begin{equation}\label{eq:jump-lower}
J_0(x,x+A)\geq \nu(A).
\end{equation}
\item[{\rm (ii)}] Assumption {\rm \ref{assum:regular}} is fulfilled for 
the measure $J({\rm d}x \, {\rm d}y)=c(x,y)J_0(x,{\rm d}y){\rm d}x$ 
with the kernel $J(x,{\rm d}y)=c(x,y)J_0(x,{\rm d}y)$.  
\end{enumerate}
\end{assum}

By assumption, $\nu$ is the L\'evy measure of a rotationally invariant  L\'evy process on ${\mathbb R}^d$. 
The function $c(x,y)$ plays the role of the coefficient for $({\cal E},{\cal F})$. 
We note here that if $B\in {\cal B}(X)$ satisfies $J(B\times B^c)=J(B^c\times B)=0$, 
then $m(B)=0$ or $m(B^c)=0$ holds. 
Hence $({\cal E},{\cal F})$ is irreducible by \cite[Theorem 1.2]{O92}.

Let 
\begin{equation}\label{eq:coeff-a}
c_1(x,y)=(a_1(x)+a_1(y)){\bf 1}_{\{|x-y|<1\}}, \quad c_2(x,y)=(a_2(x)+a_2(y)){\bf 1}_{\{|x-y|\geq 1\}}.
\end{equation}
Then under Assumption \ref{assum:lower}, 
$$J^{(i)}({\rm d}x \, {\rm d}y)=c_i(x,y)J_0(x,{\rm d}y)\,{\rm d}x, \quad i=1,2.$$
Since $({\cal E},{\cal F})$ is irreducible, 
we see by \cite[Example 4.1 and Theorems 1.2  and 2.1 (ii)]{O92} that 
$({\cal E}^{(i)},{\cal F}^{(i)})$ is also irreducible for each $i=1,2$ 
so that it is recurrent or transient by Lemma \ref{lem:irreducible}. 
Note that by Theorem \ref{thm:comparison}, 
$({\cal E},{\cal F})$ is transient if so is either of $({\cal E}^{(i)},{\cal F}^{(i)})$.

\subsection{Small jump part}
In this subsection, we prove 
\begin{thm}\label{thm:tran-sm} 
If there exist $c>0$ and  $p>(2-d)/2$ such that 
\begin{equation*}
a_1(x)\geq c(1+|x|^2)^p \quad \text{for any $x\in {\mathbb R}^d$,}
\end{equation*}
then $({\cal E}^{(1)},{\cal F}^{(1)})$ is transient.
\end{thm}

On account of Theorem \ref{thm:comparison},
we may assume that the equality holds in
\eqref{eq:jump-lower} and   
$a_1(x)=(1+|x|^2)^p$ with some $p\in ((2-d)/2,1]$. 
For the proof of Theorem \ref{thm:tran-sm}, 
we apply the Lyapunov method 
(see, e.g., \cite[Section 6]{P95} for diffusion processes 
with singular diffusion coefficients) 
by using the martingale characterization 
of the generator of $({\cal E}^{(1)},{\cal F}^{(1)})$.

In what follows, we assume that the equality holds in
\eqref{eq:jump-lower} and $a_1(x)=(1+|x|^2)^p$ for some $p\in {\mathbb R}$. 
We first apply a test function to the generator. 
Let $\langle \cdot,\cdot\rangle$ be the standard inner product on ${\mathbb R}^d$. 
Let ${\cal L}^{(1)}$ be a linear operator on $C_0^{\infty}({\mathbb R}^d)$ defined by 
\begin{equation}\label{eq:gene}
\begin{split}
{\cal L}^{(1)}u(x)
&=\int_{0<|z|<1}\left(u(x+z)-u(x)-\langle \nabla u(x), z\rangle \right)
(a_1(x)+a_1(x+z))\,\nu({\rm d}z)\\
&+\frac{1}{2}\int_{0<|z|<1}\langle\nabla u(x),z\rangle(a_1(x+z)-a_1(x-z))\,\nu({\rm d}z), 
\quad (u\in C_0^{\infty}({\mathbb R}^d)).
\end{split}
\end{equation}
Then by \cite[Theorem 2.2]{SU14}, we know that for any $u\in C_0^{\infty}({\mathbb R}^d)$ and $v\in {\cal F}$,
\begin{equation*}
{\cal E}^{(1)}(u,v)=(-{\cal L}^{(1)}u,v)_{L^2({\mathbb R}^d)}.
\end{equation*}
In particular, ${\cal L}^{(1)}u\in L^2({\mathbb R}^d)\cap {\cal B}_b({\mathbb R}^d)$. 
For any $u\in C_b^2({\mathbb R}^d)$, 
all the integrals in  \eqref{eq:gene} are convergent  
so that we can define ${\cal L}^{(1)}u$ by the expression \eqref{eq:gene}.

Here we take $\psi_{\delta}(x)=(1+|x|^2)^{-\delta} \ (\delta>0)$ as a test function. 
Note that  $\psi_{\delta}\in C_b^2({\mathbb R}^d)$.
\begin{prop}\label{prop:liyap}
If $p>(2-d)/2$, 
then for any $\delta\in (0, p-(2-d)/2)$, 
there exist positive constants $C$ and  $M$ such that 
\begin{equation*}
{\cal L}^{(1)}\psi_{\delta}(x)\leq -\frac{C\psi_{\delta}(x)}{(1+|x|^2)^{1-p}}, \quad |x|\geq M.
\end{equation*}
\end{prop}

\begin{proof}
Let 
\begin{equation}\label{eq:test-s}
{\cal L}^{(a)}\psi_{\delta}(x)
=\int_{0<|z|<1}(\psi_{\delta}(x+z)-\psi_{\delta}(x)-\langle\nabla \psi_{\delta}(x),z\rangle)
(a_1(x)+a_1(x+z))\,\nu({\rm d}z).
\end{equation}
We estimate the right hand side above 
in a way similar to \cite[Proposition 2.2]{SW20+}. 
Applying the Taylor theorem to $\psi_{\delta}$, we have 
\begin{equation*}
\begin{split}
&{\cal L}^{(a)}\psi_{\delta}(x)\\
&=4\delta(\delta+1)\int_0^1 
\left(\int_{0<|z|<1}\frac{\langle x+sz,z\rangle^2}{(1+|x+sz|^2)^{\delta+2}}
(a_1(x)+a_1(x+z))\,\nu({\rm d}z)\right)(1-s)\,{\rm d}s\\
&-2\delta\int_0^1 \left(\int_{0<|z|<1}\frac{|z|^2}{(1+|x+sz|^2)^{\delta+1}}
(a_1(x)+a_1(x+z))\,\nu({\rm d}z)\right)(1-s)\,{\rm d}s\\
&={\rm (I)}-{\rm (II)}.
\end{split}
\end{equation*}
Since $a_1(x)=(1+|x|^2)^p$, 
there exists $M>0$ for any $\varepsilon>0$ such that 
for any $x,z\in {\mathbb R}^d$ with $|x|\geq M$, $|z|<1$ and $s\in (0,1)$, 
\begin{equation}\label{eq:est-0}
\frac{2(1-\varepsilon)\psi_{\delta}(x)}{(1+|x|^2)^{2-p}}\leq 
\frac{a_1(x)+a_1(x+z)}{(1+|x+sz|^2)^{\delta+2}}
\leq\frac{2(1+\varepsilon)\psi_{\delta}(x)}{(1+|x|^2)^{2-p}}.
\end{equation}

Let 
\begin{equation*}
c_*=\int_{0<|z|<1}|z|^2\,\nu({\rm d}z), \quad c_{**}=\int_{0<|z|<1}|z|^4\,\nu({\rm d}z). 
\end{equation*}
Since the measure $\nu$ is rotationally invariant by assumption, 
we have 
$$\int_{0<|z|<1}\langle x,z \rangle^2\,\nu({\rm d}z)
=\frac{c_*}{d}|x|^2, \quad \int_{0<|z|<1}\langle x,z \rangle|z|^2\,\nu({\rm d}z)=0$$
so that 
\begin{equation*}
\int_{0<|z|<1}\langle x+sz,z\rangle^2\,\nu({\rm d}z)
=\frac{c_*}{d}|x|^2+c_{**}s^2.
\end{equation*}
Hence by \eqref{eq:est-0}, 
we obtain  for any $x\in {\mathbb R}^d$ with $|x|\geq M$,
\begin{equation*}
\begin{split}
{\rm (I)}
&\leq 
4\delta(\delta+1)\frac{2(1+\varepsilon)\psi_{\delta}(x)}{(1+|x|^2)^{2-p}}
\int_0^1 
\left(\int_{0<|z|<1}\langle x+sz,z\rangle^2
\,\nu({\rm d}z)\right)(1-s)\,{\rm d}s\\
&=4\delta(\delta+1)\frac{(1+\varepsilon)\psi_{\delta}(x)}{(1+|x|^2)^{2-p}}
\left(\frac{c_*}{d}|x|^2+\frac{c_{**}}{6}\right)
\leq \frac{4c_*\delta(\delta+1)}{d}\frac{(1+2\varepsilon)\psi_{\delta}(x)}{(1+|x|^2)^{1-p}}.
\end{split}
\end{equation*}
In the same way, we have
\begin{equation*}
{\rm (II)}
\geq 
\frac{2c_*\delta(1-\varepsilon)\psi_{\delta}(x)}{(1+|x|^2)^{1-p}}
\end{equation*}
so that 
\begin{equation}\label{eq:bound-s}
{\cal L}^{(a)}\psi_{\delta}(x)
={\rm (I)}-{\rm (II)}
\leq \frac{2\delta \psi_{\delta}(x)}{(1+|x|^2)^{1-p}}
\frac{c_*}{d}
\left\{2(\delta+1)(1+2\varepsilon)-d(1-\varepsilon)\right\}.
\end{equation}

Let 
\begin{equation*}
\begin{split}
{\cal L}^{(b)}\psi_{\delta}(x)
&=\frac{1}{2}\int_{0<|z|<1}\langle\nabla \psi_{\delta}(x),z\rangle
(a_1(x+z)-a_1(x-z))\,\nu({\rm d}z)\\
&=-\frac{\delta}{(1+|x|^2)^{\delta+1}}
\int_{0<|z|<1}\langle x,z\rangle
(a_1(x+z)-a_1(x-z))\,\nu({\rm d}z).
\end{split}
\end{equation*}
Then by the Taylor theorem, 
there exists $\theta\in (0,1)$ for any $x,z\in {\mathbb R}^d$ such that 
\begin{equation*}
\begin{split}
a_1(x+z)-a_1(x-z)
&=(1+|x+z|^2)^p-(1+|x-z|^2)^p\\
&=4p(1+|x+2\theta z|^2)^{p-1}\langle x+2\theta z,z\rangle,
\end{split}
\end{equation*}
and therefore,
\begin{equation}\label{eq:d-taylor}
{\cal L}^{(b)}\psi_{\delta}(x)
=-\frac{4p\delta}{(1+|x|^2)^{\delta+1}}
\int_{0<|z|<1}\left(\langle x,z \rangle^2+2\theta \langle x,z \rangle |z|^2\right)
(1+|x+2\theta z|^2)^{p-1}\,\nu({\rm d}z).
\end{equation}
In a similar way to \eqref{eq:est-0}, 
there exists $M>0$ for any $\varepsilon>0$ such that for any $x\in {\mathbb R}^d$ with $|x|\geq M$, 
\begin{equation}\label{eq:int-comp}
\frac{(1-\varepsilon)c_*}{d}(1+|x|^2)^p
\leq \int_{0<|z|<1}\langle x,z \rangle^2(1+|x+2\theta z|^2)^{p-1}\,\nu({\rm d}z)
\leq \frac{(1+\varepsilon)c_*}{d}(1+|x|^2)^p.
\end{equation}
We also see that there exists $c>0$ such that for any $x\in {\mathbb R}^d$ with $|x|\geq M$, 
\begin{equation*}
\left|\int_{0<|z|<1}2\theta \langle x,z \rangle |z|^2(1+|x+2\theta z|^2)^{p-1}\,\nu({\rm d}z)\right|
\leq c(1+|x|^2)^{p-1/2}.
\end{equation*}
Combining this with \eqref{eq:int-comp}, we see from \eqref{eq:d-taylor} that 
\begin{equation}\label{eq:bound-d}
{\cal L}^{(b)}\psi_{\delta}(x)
\leq 4(-1\pm2\varepsilon) \frac{pc_*\delta}{d}\frac{\psi_{\delta}(x)}{(1+|x|^2)^{1-p}}, \quad |x|\geq M.
\end{equation}
Here the sign of $\pm$ above is the same with that of $p$.

By \eqref{eq:bound-s} and \eqref{eq:bound-d},  we have for any $x\in {\mathbb R}^d$ with $|x|\geq M$, 
\begin{equation}\label{eq:upper-1}
\begin{split}
{\cal L}^{(1)}\psi_{\delta}(x)
&={\cal L}^{(a)}\psi_{\delta}(x)+{\cal L}^{(b)}\psi_{\delta}(x)\\
&\leq -\frac{2\delta \psi_{\delta}(x)}{(1+|x|^2)^{1-p}}\frac{c_*}{d}
\left\{-2(\delta+1)(1+2\varepsilon)+d(1-\varepsilon)-2p(-1\pm 2\varepsilon)\right\}.
\end{split}
\end{equation}  
Then for any $\delta\in (0, p-(2-d)/2)$, there exists $\varepsilon_0>0$ such that 
for any $\varepsilon\in (0,\varepsilon_0)$,
\begin{equation*}
-2(\delta+1)(1+2\varepsilon)+d(1-\varepsilon)-2p(-1\pm 2\varepsilon)>0.
\end{equation*}  
Hence the proof is complete by \eqref{eq:upper-1}.
\end{proof}

We next present a lemma which will be used to justify the martingale characterization 
of the generator applied to functions in ${\cal A}$. 
Let $\gamma(x)=(4+|x|^2)^{1/2}$. 
For $u\in C_b^2({\mathbb R}^d)$, let $H(u)$ be the Hessian matrix of $u$. 
We say that a function $u\in C_b^2({\mathbb R}^d)$ belongs to the class ${\cal A}$ ($u\in {\cal A}$ in notation) 
if there exist $c_1>0$ and $c_2>0$ such that for any $z\in {\mathbb R}^d$,
\begin{equation*}
\sup_{x\in {\mathbb R}^d}\left(\gamma(x)|\langle\nabla u(x),z\rangle|\right)\leq c_1|z| 
\end{equation*}
and
\begin{equation*}
\sup_{x\in {\mathbb R}^d}\left(\gamma(x)^2|\langle H(u)(x)z,z\rangle|\right)\leq c_2|z|^2.
\end{equation*}
Note that for any $\delta>0$, 
the function $\psi_{\delta}$ belongs to the class ${\cal A}$.

\begin{lem}\label{lem:cv-ge}
There exists a sequence 
$\{\phi_n\}\subset C_0^{\infty}({\mathbb R}^d)$ 
such that if $p\leq 1$, then for any $f\in {\cal A}$,
\begin{equation}\label{eq:bdd-0}
\sup_{n\in {\mathbb N}}\|{\cal L}^{(1)}(f\phi_n)\|_{\infty}<\infty
\end{equation}
and 
\begin{equation}\label{eq:conv-pt-0}
{\cal L}^{(1)}(f\phi_n)(x)\rightarrow ({\cal L}^{(1)}f)(x) \ (n\rightarrow\infty) \quad \text{for any $x\in {\mathbb R}^d$}. 
\end{equation}
\end{lem}

\begin{proof} 
Let $w\in C_0^{\infty}({\mathbb R})$ satisfy $0\leq w(t)\leq 1$ for any $t\in {\mathbb R}$ 
and 
\begin{equation*}
w(t)=\begin{cases}
1 & (|t|\leq 1),\\
0 & (|t|\geq 2).
\end{cases}
\end{equation*}
For $n\in {\mathbb N}$, let $\phi_n(x)=w(|x|/n) \ (x\in {\mathbb R}^d)$.
Then for any $f\in {\cal A}$ and $n\in {\mathbb N}$, 
we have $f\phi_n\in C_0^{\infty}({\mathbb R}^d)$ and 
\begin{equation}\label{eq:product}
\begin{split}
f\phi_n(x+z)-f\phi_n(x)-\langle\nabla(f\phi_n)(x),z\rangle
&=f(x+z)\left(\phi_n(x+z)-\phi_n(x)-\langle\nabla\phi_n(x),z\rangle\right)\\
&+\phi_n(x)\left(f(x+z)-f(x)-\langle\nabla f(x),z\rangle\right)\\
&+\langle\nabla \phi_n(x),z\rangle (f(x+z)-f(x)).
\end{split}
\end{equation}
In the same way to \cite[Proof of Theorem 2.3]{SU14}, 
there exist $c_1>0$ and $c_2>0$ such that 
for any $n\in {\mathbb N}$ and $z\in {\mathbb R}^d$,
\begin{equation*}
\sup_{x\in {\mathbb R}^d}\left(\gamma(x)|\langle\nabla \phi_n(x),z\rangle|\right)\leq c_1|z| 
\end{equation*}
and
\begin{equation*}
\sup_{x\in {\mathbb R}^d}\left(\gamma(x)^2|\langle H(\phi_n)(x)z,z\rangle|\right)\leq c_2|z|^2.
\end{equation*}
Hence $\phi_n\in {\cal A}$ for any $n\in {\mathbb N}$. 
Then by \eqref{eq:product} and the Taylor theorem, 
there exists $c_3>0$ such that for any $n\in {\mathbb N}$ and 
for any $x,z\in {\mathbb R}^d$ with  $0<|z|<\gamma(x)/2$, 
\begin{equation}\label{eq:deriv}
\begin{split}
&|f\phi_n(x+z)-f\phi_n(x)-\langle\nabla(f\phi_n)(x),z\rangle|\\
&\leq |f(x+z)\left(\phi_n(x+z)-\phi_n(x)-\langle\nabla\phi_n(x),z\rangle\right)|\\
&+|\phi_n(x)\left(f(x+z)-f(x)-\langle\nabla f(x),z\rangle\right)|\\
&+|\langle\nabla \phi_n(x),z\rangle (f(x+z)-f(x))|\\
&\leq c_3|z|^2/\gamma(x)^2.
\end{split}
\end{equation}
This in particular implies that for any $n\in {\mathbb N}$ and $x\in {\mathbb R}^d$,
\begin{equation*}
\begin{split}
&\int_{0<|z|<1}\left|f\phi_n(x+z)-f\phi_n(x)-\langle \nabla f\phi_n(x), z\rangle \right|
(a_1(x)+a_1(x+z))\,\nu({\rm d}z)\\
&\leq \frac{c_4a_1(x)}{\gamma(x)^2}\int_{0<|z|<1}
|z|^2\,\nu({\rm d}z)
\leq c_5(1+|x|^2)^{p-1}\leq c_6.
\end{split}
\end{equation*}

By the Taylor theorem again, 
there exists $c_6>0$ such that  
\begin{equation*}
|a_1(x+z)-a_1(x-z)|\leq c_6(1+|x|^2)^{p-1/2}|z|, \quad x\in {\mathbb R}^d, \ 0<|z|<1.
\end{equation*}
Since there exists $c_7>0$ such that for any  $n\in {\mathbb N}$, 
\begin{equation}\label{eq:grad}
|\langle\nabla(f\phi_n)(x),z\rangle|\leq c_7|z|/\gamma(x), \quad 
x,z\in {\mathbb R}^d,
\end{equation}
we have 
\begin{equation*}
\begin{split}
\int_{0<|z|<1}\left|\langle\nabla (f\phi_n)(x),z\rangle(a_1(x+z)-a_1(x-z))\right|\,\nu({\rm d}z)
&\leq c_8(1+|x|^2)^{p-1}\leq c_9.
\end{split}
\end{equation*}
so that \eqref{eq:bdd-0} holds. 
Then \eqref{eq:conv-pt-0} follows 
by the dominated convergence theorem. 
\end{proof}

\begin{rem}\label{rem:decay}\rm 
Let $p\leq 1$. 
Then according to the calculation in the proof of Lemma \ref{lem:cv-ge}, 
there exists $c>0$ for any $f\in {\cal A}$ such that  
$|({\cal L}^{(1)}f)(x)|\leq c(1+|x|^2)^{p-1}$ for any $x\in{\mathbb R}^d$. 
\end{rem}

\begin{proof}[Proof of Theorem {\rm \ref{thm:tran-sm}}]
Assume that the equality holds in
\eqref{eq:jump-lower} and   
$a_1(x)=(1+|x|^2)^p$ with some $p\in ((2-d)/2,1]$. 
Let ${\mathbf M}=(\{X_t\}_{t\geq 0},\{P_x\}_{x\in {\mathbb R}^d})$ 
be a symmetric Hunt process on ${\mathbb R}^d$ 
generated by $({\cal E}^{(1)},{\cal F}^{(1)})$. 
Then ${\mathbf M}$ is conservative by \cite[Theorem 2.3]{SU14}.

According to the proof of \cite[Theorem 4.3]{FU12}, 
there exists a properly exceptional set $N\subset {\mathbb R}^d$ such that 
for any $u\in C_0^{\infty}({\mathbb R}^d)$, 
\begin{equation*}
E_x\left[u(X_t)\right]-u(x)= E_x\left[\int_0^t{\cal L}^{(1)}u(X_s)\,{\rm d}s\right], 
\quad t\geq 0, \ x\in {\mathbb R}^d\setminus N.
\end{equation*}
Moreover, for any $u\in {\cal A}$, this equality is valid  by Lemma \ref{lem:cv-ge} 
so that  a stochastic process 
\begin{equation*}
M_t^{[u]}=u(X_t)-u(X_0)-\int_0^t {\cal L}^{(1)}u(X_s)\,{\rm d}s, \quad t\geq 0
\end{equation*}
is a martingale.

For $r>0$, let $B(r)=\left\{y\in {\mathbb R}^d \mid |y|<r\right\}$ 
and $\sigma_r=\inf\{t\geq 0 \mid |X_t|\leq r\}$  
the hitting time to $\overline{B(r)}$ of the symmetric Hunt process 
${\mathbf M}=(\{X_t\}_{t\geq 0},\{P_x\}_{x\in {\mathbb R}^d\setminus N})$.
Recall that for any $\delta>0$, 
the function $\psi_{\delta}=(1+|x|^2)^{-\delta}$ belongs to the class ${\cal A}$ 
so that $\|{\cal L}^{(1)}\psi_{\delta}\|_{\infty}<\infty$ by Remark \ref{rem:decay}. 
Hence by the optional stopping theorem 
applied to the martingale $(M_t^{[\psi_{\delta}]})_{t\geq 0}$, we obtain
\begin{equation}\label{eq:opt-stop}
E_x\left[\psi_{\delta}(X_{t\wedge \sigma_r})\right]
=\psi_{\delta}(x)
+E_x\left[\int_0^{t\wedge \sigma_r}{\cal L}^{(1)}\psi_{\delta}(X_s)\,{\rm d}s\right].
\end{equation}

Let $\delta\in (0,p-(2-d)/2)$. 
Then by Proposition \ref{prop:liyap}, 
there exists $r>0$ such that 
${\cal L}^{(1)}\psi_{\delta}(z)\leq 0$ for any $z\in B(r)^c$.  
Fix $R>r$ and $x\in B(R)^c$. Since
\begin{equation*}
E_x\left[\psi_{\delta}(X_{t\wedge \sigma_r})\right]
\geq E_x\left[\psi_{\delta}(X_{\sigma_r}); \sigma_r\leq t\right]
\geq \left(\inf_{y\in \overline{B(r)}}\psi_{\delta}(y)\right)P_x(\sigma_r\leq t)
\end{equation*}
and
\begin{equation*}
\begin{split}
\psi_{\delta}(x)
+E_x\left[\int_0^{t\wedge \sigma_r}{\cal L}^{(1)}\psi_{\delta}(X_s)\,{\rm d}s\right]
=\psi_{\delta}(x)
+E_x\left[\int_0^{t\wedge \sigma_r}{\cal L}^{(1)}\psi_{\delta}(X_{s-})\,{\rm d}s\right]
\leq \sup_{w\in B(R)^c}\psi_{\delta}(w),
\end{split}
\end{equation*}
we have by \eqref{eq:opt-stop},
\begin{equation*}
P_x(\sigma_r\leq t)
\leq \frac{\sup_{w\in B(R)^c}\psi_{\delta}(w)}{\inf_{y\in \overline{B(r)}}\psi_{\delta}(y)}, 
\quad x\in B(R)^c.
\end{equation*}
Letting $t\rightarrow\infty$, we get 
\begin{equation*}
P_x(\sigma_r<\infty)
\leq \frac{\sup_{w\in B(R)^c}\psi_{\delta}(w)}{\inf_{y\in \overline{B(r)}}\psi_{\delta}(y)}<1, \quad x\in B(R)^c
\end{equation*}
so that $({\cal E}^{(1)}, {\cal F}^{(1)})$ is not recurrent. 
Hence it is transient by irreducibility and Lemma \ref{lem:irreducible}.
Theorem \ref{thm:comparison} further implies that  
$({\cal E}^{(1)}, {\cal F}^{(1)})$ is transient for any $p>(2-d)/2$ under Assumption \ref{assum:lower}.  
\end{proof}

\subsection{Big jump part}\label{subsect:big}
In this subsection, we assume the next condition 
on the measure $\nu$ in Assumption \ref{assum:lower}:
\begin{assum}
The measure $J({\rm d}x\,{\rm d}y)$ 
satisfies Assumption {\rm \ref{assum:lower}} with 
$$\nu({\rm d}z)={\bf 1}_{\{|z|\geq 1\}}\frac{{\rm d}z}{|z|^{d+\alpha}}$$ 
for some $\alpha\in (0,2)$. 
\end{assum}

This assumption says that $\nu$ is the restriction on $\{|z|\geq 1\}$ of 
the L\'evy measure of a symmetric $\alpha$-stable process up to constant. 
We now prove 
\begin{thm}\label{thm:tran-bg} 
If there exist $c>0$ and $q>(\alpha-d)/2$ such that 
$$a_2(x)\geq c(1+|x|^2)^q, \quad x\in {\mathbb R}^d,$$
then $({\cal E}^{(2)},{\cal F}^{(2)})$ is transient.
\end{thm}

In the similar way as Theorem \ref{thm:tran-sm},  
we prove Theorem \ref{thm:tran-bg} by the Lyapunov method;
however, in order to ensure the existence of  the test function,  
we carry out elementary but involved calculation. 

On account of Theorem \ref{thm:comparison}, we may and do assume that the equality holds in \eqref{eq:jump-lower} and 
$a_2(x)=(1+|x|^2)^q$ for some $q<\alpha/2$. 
Let ${\cal L}^{(2)}$ be a linear operator on $C_0^{\infty}({\mathbb R}^d)$ defined by 
\begin{equation}\label{eq:big-gene} 
{\cal L}^{(2)}u(x)
=\int_{|z|\geq 1}\left(u(x+z)-u(x)\right)
\frac{a_2(x)+a_2(x+z)}{|z|^{d+\alpha}}\,{\rm d}z.
\end{equation}
Then by \cite[Theorem 2.2]{SU14} again, 
we know that for any $u\in C_0^{\infty}({\mathbb R}^d)$ and $v\in {\cal F}^{(2)}$,
\begin{equation}
{\cal E}^{(2)}(u,v)=(-{\cal L}^{(2)}u,v)_{L^2({\mathbb R}^d)}.
\end{equation}
In particular, 
we can also define ${\cal L}^{(2)}u$ for $u\in C_b^2({\mathbb R}^d)$ 
by the expression \eqref{eq:big-gene}.

\begin{lem}\label{lem:cv-big-ge}
There exists a sequence $\{\phi_n\}\subset C_0^{\infty}({\mathbb R}^d)$ such that 
if $q<\alpha/2$, then for any $f\in {\cal A}$,
\begin{equation}\label{eq:bdd}
\sup_{n\in {\mathbb N}}\|{\cal L}^{(2)}(f\phi_n)\|_{\infty}<\infty
\end{equation}
and 
\begin{equation}\label{eq:conv-pt}
{\cal L}^{(2)}(f\phi_n)(x)\rightarrow ({\cal L}^{(2)}f)(x) \ (n\rightarrow\infty) 
\quad \text{for any $x\in {\mathbb R}^d$}. 
\end{equation}
\end{lem}

\begin{proof} 
Let $\{\phi_n\}\subset C_0^{\infty}({\mathbb R}^d)$ be 
the same sequence as in the proof of Lemma \ref{lem:cv-ge}.
According to \cite[Theorem 2.2]{SU14}, 
we have  for any $u\in C_0^{\infty}({\mathbb R}^d)$, 
\begin{equation}\label{eq:gene-1}
\begin{split}
{\cal L}^{(2)}u(x)
&=\int_{1\leq |z|<\gamma(x)/2}\left(u(x+z)-u(x)-\langle \nabla u(x), z\rangle \right)
\frac{a_2(x)+a_2(x+z)}{|z|^{d+\alpha}}\,{\rm d}z\\
&+\frac{1}{2}\int_{1\leq |z|<\gamma(x)/2}\langle\nabla u(x),z\rangle
\frac{a_2(x+z)-a_2(x-z)}{|z|^{d+\alpha}}\,{\rm d}z\\
&+\int_{|z|\geq \gamma(x)/2}\left(u(x+z)-u(x)\right)
\frac{a_2(x)+a_2(x+z)}{|z|^{d+\alpha}}\,{\rm d}z. 
\end{split}
\end{equation}
This expression is valid by replacing $u$ with $f\in {\cal A}$. 
Then by \eqref{eq:deriv},
\begin{equation*}
\begin{split}
&\int_{1\leq |z|<\gamma(x)/2}\left|f\phi_n(x+z)-f\phi_n(x)-\langle \nabla f\phi_n(x), z\rangle \right|
\frac{a_2(x)+a_2(x+z)}{|z|^{d+\alpha}}\,{\rm d}z\\
&\leq \frac{c_1a_2(x)}{\gamma(x)^2}
\int_{1\leq |z|<\gamma(x)/2}|z|^{2-(d+\alpha)}\,{\rm d}z
\leq c_2(1+|x|^2)^{q-\alpha/2}.
\end{split}
\end{equation*}

By the Taylor theorem, there exists $c_3>0$ such that  
\begin{equation*}
|a_2(x+z)-a_2(x-z)|\leq c_3(1+|x|^2)^{q-1/2}|z|, \quad 
x\in {\mathbb R}^d, \ 1\leq |z|<\gamma(x)/2.
\end{equation*}
Then for $f\in {\cal A}$, we get by \eqref{eq:grad},
$$
\int_{1\leq |z|<\gamma(x)/2}
\left|\langle\nabla (f\phi_n)(x),z\rangle
\frac{a_2(x+z)-a_2(x-z)}{|z|^{d+\alpha}}\right|\,{\rm d}z
\leq c_4(1+|x|^2)^{q-\alpha/2}.
$$

Let 
$$F_q(x)=
\begin{cases}
(1+|x|^2)^{q-\alpha/2}, & -d/2<q<\alpha/2, \\ 
(1+|x|^2)^{-(d+\alpha)/2}\log(2+|x|), & q=-d/2, \\
(1+|x|^2)^{-(d+\alpha)/2}, & q<-d/2.\end{cases}
$$
If $|z|<2\gamma(x)$, then $|x+z|\leq c(1+|x|)$ for some $c>0$ so that 
\begin{equation}\label{eq:bound-1}
\begin{split}
\int_{\gamma(x)/2\leq |z|<2\gamma(x)}\frac{a_2(x+z)}{|z|^{d+\alpha}}\,{\rm d}z
&\leq \frac{c_5}{\gamma(x)^{d+\alpha}}
\int_{|x+z|<c(1+|x|)}a_2(x+z)\,{\rm d}z\\
&=\frac{c_5}{\gamma(x)^{d+\alpha}}
\int_{|z|<c(1+|x|)}a_2(z)\,{\rm d}z
\leq c_6F_q(x).
\end{split}
\end{equation}
On the other hand, if $|z|\geq 2\gamma(x)$, then 
$c(1+|x|)\leq |x+z|\leq |x|+|z|$ for some $c>0$, which yields 
\begin{equation*}
\int_{|z|\geq 2\gamma(x)}\frac{a_2(x+z)}{|z|^{d+\alpha}}\,{\rm d}z
\leq c_7\int_{|z|\geq 2\gamma(x)}\frac{a_2(x)+a_2(z)}{|z|^{d+\alpha}}\,{\rm d}z
\leq c_8(1+|x|^2)^{q-\alpha/2}.
\end{equation*}
Combining this with \eqref{eq:bound-1}, 
we obtain 
\begin{equation*}
\begin{split}
&\int_{|z|\geq \gamma(x)/2}\left|f\phi_n(x+z)-f\phi_n(x)\right|
\frac{a_2(x)+a_2(x+z)}{|z|^{d+\alpha}}\,{\rm d}z\\
&\leq c_9\left(a_2(x)\int_{|z|\geq \gamma(x)/2}\frac{{\rm d}z}{|z|^{d+\alpha}}
+\int_{|z|\geq \gamma(x)/2}\frac{a_2(x+z)}{|z|^{d+\alpha}}\,{\rm d}z\right)
\leq c_{10}F_q(x)
\end{split}
\end{equation*}
so that \eqref{eq:bdd} holds. 
We also have \eqref{eq:conv-pt} by \eqref{eq:gene-1} and the dominated convergence theorem. 
\end{proof}

\begin{rem}\rm 
We see from the proof of Lemma \ref{lem:cv-big-ge} 
that for any $f\in {\cal A}$, there exists $c>0$ such that 
$$|({\cal L}^{(2)}f)(x)|\leq cF_q(x), \quad x\in {\mathbb R}^d.$$
\end{rem}

\begin{prop}\label{prop:liyap-big}
Let $q\in ((\alpha-d)/2,\alpha/2)$. 
Then there exists $\delta_0>0$ such that for any $\delta\in (0,\delta_0)$, 
there exist $C>0$ and $M>0$ such that 
\begin{equation}\label{eq:b-neg}
{\cal L}^{(2)}\psi_{\delta}(x)
\leq \frac{-C\psi_{\delta}(x)}{(1+|x|^2)^{\alpha/2-q}}, \quad |x|\geq M.
\end{equation}
\end{prop}

To show Proposition \ref{prop:liyap-big}, 
we first calculate ${\cal L}^{(2)}\psi_{\delta}(x)$ for $\delta>0$. 
In what follows, we present the proof of Proposition \ref{prop:liyap-big} 
only for $d\geq 2$ and $q\in ((\alpha-d)/2,0)$, 
but the same argument applies for other cases. 

Assume that $d\geq 2$ and $q<0$. 
Let $\langle x\rangle=\sqrt{1+|x|^2}$ for $x\in {\mathbb R}^d$. 
We define for $\delta\geq 0$, $x\in {\mathbb R}^d$, $r\geq 0$ and $s\in [0,1]$,
\begin{equation*}
\begin{split}
J_q(\delta,x,r,s)
&=\left[\frac{1}{(1+r^2+2|x|sr/\langle x\rangle)^{\delta}}-1\right]
\left[1+\left(1+r^2+2|x|sr/\langle x\rangle\right)^q\right]\\
&+\left[\frac{1}{(1+r^2-2|x|sr/\langle x\rangle)^{\delta}}-1\right]
\left[1+\left(1+r^2-2|x|sr/\langle x\rangle\right)^q\right].
\end{split}
\end{equation*}
Let $\omega_0=2$, and 
let $\omega_{d-2}$ be the area of the $d-2$ dimensional unit surface for $d\geq 3$. 
Then for $\delta>0$ and $x\in {\mathbb R}^d$,
\begin{equation}\label{eq:test-b}
\begin{split}
{\cal L}^{(2)}\psi_{\delta}(x)
&=\int_{|z|\geq 1}(\psi_{\delta}(x+z)-\psi_{\delta}(x))\frac{a_2(x)+a_2(x+z)}{|z|^{d+\alpha}}\,{\rm d}z\\
&=\psi_{\delta}(x)a_2(x)\int_{|z|\geq 1}\left[\frac{\psi_{\delta}(x+z)}{\psi_{\delta}(x)}-1\right]
\frac{1+a_2(x+z)/a_2(x)}{|z|^{d+\alpha}}\,{\rm d}z\\
&=\psi_{\delta}(x)a_2(x)
\int_{|z|\geq 1}
\left[\frac{(1+|x|^2)^{\delta}}{(1+|x+z|^2)^{\delta}}-1\right]
\frac{1+(1+|x+z|^2)^q/(1+|x|^2)^q}{|z|^{d+\alpha}}\,{\rm d}z\\
&=\frac{\omega_{d-2}\psi_{\delta}(x)}{(1+|x|^2)^{\alpha/2-q}}
\int_{1/{\sqrt{1+|x|^2}}}^{\infty}
\left(\int_0^1 J_q(\delta,x,r,s)(1-s^2)^{(d-3)/2}\,{\rm d}s\right)\frac{{\rm d}r}{r^{1+\alpha}}.
\end{split}
\end{equation}
At the fourth equality above
we used the change of variables formula ($z=\langle x\rangle u$) 
and then the polar coordinate expression. 

Let 
$$F_q(\delta,x)=\int_{1/{\sqrt{1+|x|^2}}}^{\infty}
\left(\int_0^1 J_q(\delta,x,r,s)(1-s^2)^{(d-3)/2}\,{\rm d}s\right)\frac{{\rm d}r}{r^{1+\alpha}}, 
\quad \delta\geq 0, \ x\in {\mathbb R}^d.$$
Let $\tilde{J}_q(\delta,r,s)=\lim_{|x|\rightarrow\infty}J_q(\delta,x,r,s)$ and 
$$\Lambda_q(\delta)=\int_0^{\infty}
\left(\int_0^1 \tilde{J}_q(\delta,r,s)(1-s^2)^{(d-3)/2}\,{\rm d}s\right)\frac{{\rm d}r}{r^{1+\alpha}}, 
\quad \delta\geq 0.$$
Note that 
\begin{equation*}
\begin{split}
\tilde{J}_q(\delta,r,s)
&=\left[\frac{1}{(r^2+2rs+1)^{\delta}}-1\right](1+(r^2+2rs+1)^q)\\
&+\left[\frac{1}{(r^2-2rs+1)^{\delta}}-1\right](1+(r^2-2rs+1)^q), \quad \delta\geq 0, \ r\geq 0, \ s\in [0,1).
\end{split}
\end{equation*}
We then define
\begin{equation*}
K_q(\delta,r,s)=
\frac{\log(r^2+2rs+1)}{(r^2+2rs+1)^{\delta}}[1+(r^2+2rs+1)^q], 
\quad \delta\in {\mathbb R}, \ r\geq 0, \ s\in (-1,1).
\end{equation*}

\begin{lem}\label{lem:derivative}
Let $d\geq 2$ and $q\in [(\alpha-d)/2,0)$.
\begin{enumerate}
\item[\rm (1)] There exists $\delta_0>0$ such that for any $\delta\in [0,\delta_0)$, 
\begin{equation}\label{eq:lim}
F_q(\delta,x)\rightarrow \Lambda_q(\delta) \quad \text{as $|x|\rightarrow\infty$.}
\end{equation}

\item[{\rm (2)}] 
Let $\delta_0>0$ be as in {\rm (1)} and $\delta\in [0,\delta_0)$. Then 
\begin{equation}\label{eq:lim-1}
\Lambda_q'(\delta)=-\int_0^{\infty}\left(\int_{-1}^1K_q({\delta},r,s)(1-s^2)^{(d-3)/2}\,{\rm d}s\right)
\frac{{\rm d}r}{r^{1+\alpha}}.
\end{equation}
\item[{\rm (3)}] If $q=(\alpha-d)/2$, then $\Lambda_q'(0)=0$.
\end{enumerate}
\end{lem}

We postpone the proof of Lemma \ref{lem:derivative} to
the next section and prove Proposition {\rm \ref{prop:liyap-big}}.

\begin{proof}[Proof of Proposition {\rm \ref{prop:liyap-big}}]
Let $q=(\alpha-d)/2$. 
Then $\Lambda_q(0)=\Lambda_q'(0)=0$ by Lemma \ref{lem:derivative} (2), (3). 
Moreover, $\Lambda'_q(\delta)$ is decreasing in $\delta\in [0,\delta_0)$ 
because for each $(r,s)\in [0,\infty)\times [0,1)$, 
the function $K_q(\delta,r,s)$ is decreasing in $\delta\in [0,\infty)$. 
This yields $\Lambda_q'(\delta)<0$ and thus $\Lambda_q(\delta)<0$ for any $\delta\in (0,\delta_0)$. 

By Lemma \ref{lem:derivative} (1), 
there exists $M>0$ for any $\delta\in [0,\delta_0)$ such that 
$F_q(\delta,x)\leq \Lambda_q(\delta)/2<0$ for any $x\in {\mathbb R}^d$ with $|x|\geq M$. 
Hence the proof is complete by \eqref{eq:test-b}. 
\end{proof}

\begin{proof}[Proof of Theorem {\rm \ref{thm:tran-bg}}]
By using Proposition \ref{prop:liyap-big}, 
we can follow the proof of Theorem \ref{thm:tran-sm}. 
\end{proof}

\subsection{Applications}\label{subsect:appl}
We apply Theorems \ref{thm:tran-sm} and \ref{thm:tran-bg} to concrete models. 
 
\subsubsection{Necessary and sufficient condition for recurrence}
Let $a_1$ and $a_2$ be Borel measurable functions on ${\mathbb R}^d$ such that 
for some $p,q\in {\mathbb R}$ and some positive constants $c_{11}$, $c_{12}$, $c_{21}$, $c_{22}$,
$$c_{11}(1+|x|)^p\leq a_1(x)\leq c_{12}(1+|x|)^p, 
\quad c_{21}(1+|x|)^q\leq a_2(x)\leq c_{22}(1+|x|)^q, 
\quad x\in {\mathbb R}^d.$$  
We then define $c(x,y)$ as in \eqref{eq:coeff-a}. 

For  fixed  $\alpha\in (0,2)$ and $\beta\in (0,2)$, let 
$$J(x,y)=\frac{c_1(x,y)}{|x-y|^{d+\alpha}}{\bf 1}_{\{|x-y|<1\}}
+\frac{c_2(x,y)}{|x-y|^{d+\beta}}{\bf 1}_{\{|x-y|\geq 1\}}$$
and $J(x,{\rm d}y)=J(x,y)\,{\rm d}y$.
If $q<\beta$, then Assumptions \ref{assum:lower} and \ref{assum:upper} are 
fulfilled with 
$$J_0(x,{\rm d}y)=\left(\frac{1}{|x-y|^{d+\alpha}}{\bf 1}_{\{|x-y|<1\}}
+\frac{1}{|x-y|^{d+\beta}}{\bf 1}_{\{|x-y|\geq 1\}}\right)\,{\rm d}y$$
and
$$\nu({\rm d}z)=\left(\frac{1}{|z|^{d+\alpha}}{\bf 1}_{\{|z|<1\}}
+\frac{1}{|z|^{d+\beta}}{\bf 1}_{\{|z|\geq 1\}}\right)\,{\rm d}z.$$
If $\alpha=\beta$ and $p=q=0$, 
then the symmetric Hunt process generated by $({\cal E},{\cal F})$ 
is the so-called symmetric $\alpha$-stable-like process 
introduced in \cite{CK03}.
Hence for $p\in {\mathbb R}$ and $q<\beta$, 
we can regard $({\cal E},{\cal F})$ as a Dirichlet form of symmetric stable-like 
with unbounded/degenerate coefficients.

By Theorems \ref{thm:tran-sm} and \ref{thm:tran-bg} with Example \ref{exam:rec} below, 
we obtain a necessary and sufficient condition for recurrence of $({\cal E}^{(i)},{\cal F}^{(i)}) \ (i=1,2)$. 
Then by Theorem \ref{thm:comparison} and Example \ref{exam:rec},  
we also get a necessary and sufficient condition for recurrence of $({\cal E},{\cal F})$. 
Summarizing the observations above, we have 
\begin{cor} \label{cor:tran-iff}
Let $p\in {\mathbb R}$ and $q<\beta$. 
Then under the setting in this subsection, the following assertions hold{\rm :} 
\begin{itemize}
\item[{\rm (1)}] 
$({\cal E}^{(1)},{\cal F}^{(1)})$ is recurrent if and only if $p\leq 2-d$. 
\item[{\rm (2)}] 
$({\cal E}^{(2)},{\cal F}^{(2)})$ is recurrent if and only if $q\leq \beta-d$. 
\item[{\rm (3)}] 
$({\cal E},{\cal F})$ is recurrent if and only if 
$p\leq 2-d$ and $q\leq \beta-d$. 
\end{itemize}
\end{cor}

We make a comment on transience of regular Dirichlet forms with no killing measure. 
We keep the same conditions on $J(x,{\rm d}y)$ as in Corollary \ref{cor:tran-iff}. 
Let $\{a_{ij}(x)\}_{1\leq i,j\leq d}$ be a family of Borel measurable functions on ${\mathbb R}^d$ 
such that $a_{ij}(x)=a_{ji}(x) \ (1\leq i,j\leq d)$ for any $x\in {\mathbb R}^d$,  
and for some positive constants $r$, $c_{01}$ and $c_{02}$,  
$$c_{01}(1+|x|)^r|\xi|^2\leq \sum_{i,j=1}^d a_{ij}(x)\xi_i\xi_j\leq c_{02}(1+|x|)^r|\xi|^2, \quad x,\xi\in {\mathbb R}^d.$$ 
Let $({\cal E},C_0^{\infty}({\mathbb R}^d))$ be the quadratic form on $L^2({\mathbb R}^d)$ 
defined by 
\begin{equation}\label{eq:d-j}
{\cal E}(u,u)
=\int_{{\mathbb R}^d}\sum_{i,j=1}^d a_{ij}(x)\frac{\partial u}{\partial x_i}(x)\frac{\partial u}{\partial x_j}(x)\,{\rm d}x
+\iint_{{\mathbb R}^d\times{\mathbb R}^d}(u(x)-u(y))^2\,J(x, {\rm d}y){\rm d}x.
\end{equation}
Then $({\cal E},C_0^{\infty}({\mathbb R}^d))$ is closable on $L^2({\mathbb R}^d)$ and 
its closure $({\cal E},{\cal F})$ is a regular Dirichlet form on $L^2({\mathbb R}^d)$.

If the jumping measure $J(x,{\rm d}y)\,{\rm d}x$ vanishes in \eqref{eq:d-j}, 
then $({\cal E},{\cal F})$ is recurrent if and only if $r\leq 2-d$ (\cite[Examples 1 and 2]{I78}).
Hence by Corollary \ref{cor:tran-iff} and Example \ref{exam:j-d} below, 
$({\cal E},{\cal F})$ is recurrent if and only if 
$p\leq 2-d$, $q\leq \beta-d$ and $r\leq 2-d$. 

\subsubsection{Direct product}\label{subsubsect:direct}
For $i=1,2$, let $(\tilde{\cal E}^{(i)},\tilde{\cal F}^{(i)})$ be 
a regular Dirichlet form on $L^2({\mathbb R}^d)$ as $({\cal E},{\cal F})$ 
in Corollary \ref{cor:tran-iff}. 
Let $\alpha_i$, $\beta_i$ $p_i$, $q_i$ denote 
the parameters corresponding to $\alpha$, $\beta$, $p$, $q$,
respectively. 
Assume that $q_i<\beta_i$ for $i=1,2$. 
Let $\tilde{\mathbf M}^{(i)}=(\{X_t^{(i)}\}_{t\geq 0},\{P_x^{(i)}\}_{x\in {\mathbb R}^d})$ 
be a symmetric Hunt process on ${\mathbb R}^d$ generated by $(\tilde{\cal E}^{(i)},\tilde{\cal F}^{(i)})$. 
Let $\tilde{\mathbf M}=(\{X_t\}_{t\geq 0},\{P_x\}_{x\in {\mathbb R}^{2d}})$ be 
the direct product of $\tilde{\mathbf M}^{(1)}$ and $\tilde{\mathbf M}^{(2)}$ defined by 
$$X_t=(X_t^{(1)},X_t^{(2)}), \quad P_{(x_1,x_2)}=P_{x_1}^{(1)}\otimes P_{x_2}^{(2)}, \quad x_1,x_2\in {\mathbb R}^d.$$
Then by \cite[Theorem 3.1]{O97}, 
$\tilde{\mathbf M}$ is a symmetric Markov process on ${\mathbb R}^{2d}$ 
such that the associated Dirichlet form $(\tilde{\cal E},\tilde{\cal F})$ on $L^2({\mathbb R}^{2d})$ is regular. 
Moreover,  ${\cal C}=C_0^{\infty}({\mathbb R}^d)\otimes C_0^{\infty}({\mathbb R}^d)$ is a core of $({\cal E},{\cal F})$ 
and 
$$\tilde{\cal E}(u,u)
=\int_{{\mathbb R}^d}\tilde{\cal E}^{(1)}(u(\cdot, y),u(\cdot,y))\,{\rm d}y
+\int_{{\mathbb R}^d}\tilde{\cal E}^{(2)}(u(x,\cdot),u(x,\cdot))\,{\rm d}x, \quad u\in {\cal C}.$$
Here $C_0^{\infty}({\mathbb R}^d)\otimes C_0^{\infty}({\mathbb R}^d)$ 
is the linear span of functions $u^{(1)}\otimes u^{(2)}(x,y):=u^{(1)}(x)u^{(2)}(y)$ 
for $u^{(i)}\in C_0^{\infty}({\mathbb R}^d)$.

Since each $(\tilde{\cal E}^{(i)},\tilde{\cal F}^{(i)})$ is irreducible, 
so is $(\tilde{\cal E},\tilde{\cal F})$ by \cite[Theorem 5.1]{O97}. 
Furthermore, $(\tilde{\cal E},\tilde{\cal F})$ is transient 
if so is either of $(\tilde{\cal E}^{(i)},\tilde{\cal F}^{(i)})$ 
by \cite[Theorem 5.2]{O97}. 
Hence by Corollary \ref{cor:tran-iff}, $(\tilde{\cal E},\tilde{\cal F})$ is transience 
if at least one of the following inequalities holds: $p_1\vee p_2>2-d$, $q_1>\beta_1-d$, $q_2>\beta_2-d$.

\subsubsection{Feller processes}
The proofs of Theorems \ref{thm:tran-sm} and \ref{thm:tran-bg} are applicable to 
a class of Feller processes
 (see, e.g., \cite[Definitions 1.16 and 1.24]{BSW14} 
for the definition of Feller processes and Feller generators, 
and \cite[Theorem 2.21]{BSW14} for the form of the Feller generator).
Let ${\mathbf M}=(\{X_t\}_{t\geq 0}, \{P_x\}_{x\in {\mathbb R}^d})$ 
be a Feller process on ${\mathbb R}^d$ 
such that 
\begin{itemize}
\item $P_x(X_t\in B)>0$ for any $t>0$, $x\in {\mathbb R}^d$ 
and $B\in {\cal B}({\mathbb R}^d)$ with positive Lebesgue measure; 
\item the function $x\mapsto E_x[f(X_t)]$ is continuous on ${\mathbb R}^d$ 
for each $t>0$ and bounded continuous function $f$ on ${\mathbb R}^d$.
\end{itemize}
Then by \cite[Theorems 6.22--6.24]{BSW14}, 
${\mathbf M}$ is Harris recurrent or transient in the sense of \cite[Definition 6.21]{BSW14} 
(see, e.g., \cite[Theorem 6.27]{BSW14} 
for the equivalent conditions of Harris recurrence and transience).  

Let $(L,D(L))$ be the Feller generator of ${\mathbf M}$. 
We suppose that $C_0^{\infty}({\mathbb R}^d)\subset D(L)$ 
and for any $u\in C_0^{\infty}({\mathbb R}^d)$,
\begin{equation}\label{eq:feller}
\begin{split}
Lu(x)
&=\sum_{i,j=1}^da_{ij}(x)\frac{\partial^2 u}{\partial x_i \partial x_j}(x)
+\langle l(x),\nabla u(x)\rangle\\
&+\int_{{\mathbb R}^d}(u(x+z)-u(x)-\langle\nabla u(x),z\rangle{\bf 1}_{\{0<|z|<1\}})\kappa(x,z)\nu({\rm d}z).
\end{split}
\end{equation}
Here 
$(a_{ij}(x))_{i,j=1}^d$ is a symmetric and nonnegative definite matrix, 
$l(x)$ is a ${\mathbb R}^d$-valued measurable function on ${\mathbb R}^d$, 
and $\kappa(x,z)$ is a nonnegative Borel measurable function on 
${\mathbb R}^d\times {\mathbb R}^d$.  
$\nu({\rm d}z)$ is 
a positive Borel measure on ${\mathbb R}^d$ such that 
$\int_{0<|z|<1}|z|^2\nu({\rm d}z)<\infty$ and 
${\bf 1}_{\{|z|\geq 1\}}\nu({\rm d}z)={\bf 1}_{\{|z|\geq 1\}}|z|^{-(d+\alpha)}\,{\rm d}z$ 
with some $\alpha\in (0,2)$. 

We assume the next conditions on $\kappa(x,z)$: 
\begin{itemize}
\item there exist positive constants $c_1$ and $M$ 
such that for any $x\in {\mathbb R}^d$ with $|x|\geq M$,
$$|x|^2\int_{0<|z|<1}|z|^2\kappa(x,z)\,\nu({\rm d}z)
\leq c_1\int_{0<|z|<1}\langle x,z\rangle^2\kappa(x,z)\,\nu({\rm d}z);$$
\item for some $c_2\geq 0$ and $q>(\alpha-d)/2$,  
$\kappa(x,z)=c_2((1+|x|^2)^q+(1+|x+z|^2)^q)$ 
for any $x\in {\mathbb R}^d$ and $z\in {\mathbb R}^d$ with $|z|\geq 1$. 
\end{itemize}
Let 
$$A(x)=\frac{1}{|x|^2}\left[\sum_{i,j=1}^da_{ij}(x)x_ix_j
+\int_{0<|z|<1}\langle x,z\rangle^2\kappa(x,z)\,\nu({\rm d}z)\right],$$
$$B(x)=\sum_{i=1}^d a_{ii}(x)+\int_{0<|z|<1}|z|^2\kappa(x,z)\,\nu({\rm d}z), \quad C(x)=\langle l(x),x\rangle.$$
Then ${\mathbf M}$ is transient if 
$$\liminf_{|x|\rightarrow\infty}\frac{B(x)+C(x)}{A(x)}>2.$$
We omit the proof of this assertion 
because it is almost identical with those of Theorems \ref{thm:tran-sm} and \ref{thm:tran-bg} 
by using \cite[Theorem 6.27]{BSW14}. 

For instance,  if $a_{ij}(x)\equiv 0$, $\kappa(x,z)=a_1(x)+a_1(x+z)$  
for any $x\in {\mathbb R}^d$ and $z\in {\mathbb R}^d$ with $|z|<1$ and 
$$(l(x))_i=\frac{1}{2}\int_{0<|z|<1}z_i(a_1(x+z)-a_1(x-z))\,\nu(x,{\rm d}z), \quad i=1,\dots, d,$$
then $L$ has the same form as ${\cal L}^{(1)}+{\cal L}^{(2)}$. 
However, we do not know general conditions for the operator as in \eqref{eq:feller} 
being a Feller generator in terms of the functions $a_{ij}(x)$, $\kappa(x,z)$, $l(x)$ 
and the measure $\nu({\rm d}z)$.

\section{Proof of Lemma {\ref{lem:derivative}}}\label{sect:lem}
This section is devoted to the proof of Lemma {\ref{lem:derivative}}.

\subsection{Proof of Lemma {\ref{lem:derivative}} (1) and (2)}
Fix $\delta>0$. Recall that $d\geq 2$ and $q\in [(\alpha-d)/2,0)$.
\begin{proof}[Proof of Lemma {\rm \ref{lem:derivative} (1)}]
Since $0\leq 2s|x|/\langle x \rangle\leq 1$ 
for any $x\in {\mathbb R}^d$ and $s\in [0,1)$, 
there exists $c_1>0$ such that 
for any $x\in {\mathbb R}^d$, $r\geq 2$ and $s\in [0,1)$, 
\begin{equation*}
1+\left(1+r^2\pm\frac{2s|x|}{\langle x \rangle}r\right)^q
\leq c_1
\end{equation*}
and thus
\begin{equation}\label{eq:j-1}
|J_q(\delta,x,r,s)|\leq 2c_1.
\end{equation}

If $0\leq r\leq 3/2$, then for any $x\in {\mathbb R}^d$ and $s\in [0,1)$, 
\begin{equation}\label{eq:lower-rs}
1+r^2\pm \frac{2s|x|}{\langle x \rangle}r\ge (r-1)^2\vee (1-s^2)\geq (r-1)^2,
\end{equation}
which implies that 
\begin{equation}\label{eq:q-bdd}
1+\left(1+r^2\pm\frac{2s|x|}{\langle x \rangle}r\right)^q
\leq c_2[(r-1)^2\vee (1-s^2)]^q.
\end{equation}
Let $\eta$ be a negative constant such that 
\begin{equation}\label{eq:eta}
-\frac{1}{2}<\eta<q+\frac{d-1}{2}.
\end{equation}
If $1/2\leq r\leq 3/2$, then by \eqref{eq:lower-rs} and \eqref{eq:q-bdd}, 
\begin{equation}\label{eq:j-2}
|J_q(\delta,x,r,s)|\leq 2c_2\left(1+\frac{1}{|r-1|^{2\delta}}\right)|r-1|^{2\eta}(1-s^2)^{q-\eta}.
\end{equation}

Let $r\in [0,1/2]$. Then the function 
$$g_1^{(r)}(y)=\frac{1}{(1+r^2+2ry)^{\delta}}+\frac{1}{(1+r^2-2ry)^{\delta}}, \quad y\in [0,1]$$
is increasing and therefore, 
$$\frac{2}{(1+r^2)^{\delta}}\leq g_1^{(r)}(y)
\leq \frac{1}{(1+r)^{2\delta}}+\frac{1}{(1-r)^{2\delta}}, \quad y\in [0,1].$$
In particular, there exists $c_3>0$ such that for any $y\in [0,1]$ and $r\in [0,1/2]$,
$$\left|g_1^{(r)}(y)-2\right|
\leq c_3r^2.$$
Hence by \eqref{eq:lower-rs}, there exists $c_4>0$ such that 
for any $x\in {\mathbb R}^d$, $r\in [0,1/2]$ and $s\in [0,1)$, 
\begin{equation}\label{eq:bdd0-1}
\begin{split}
&\left|
\frac{1}{(1+r^2+2s|x|r/\langle x\rangle)^{\delta}}+\frac{1}{(1+r^2-2s|x|r/\langle x\rangle)^{\delta}}-2\right|
\left[1+\left(1+r^2-\frac{2s|x|}{\langle x\rangle}r\right)^q\right]
\leq c_4r^2.
\end{split}
\end{equation}
Let
$$
g_2^{(r)}(y)
=\left[1-\frac{1}{(1+r^2+2ry)^{\delta}}\right]
\left[(1+r^2-2ry)^q-(1+r^2+2ry)^q\right], \quad y\in [0,1].
$$
Since $g_2^{(r)}$ is a increasing, there exists $c_5>0$ such that for any $y\in [0,1]$ and $r\in [0,1/2]$, 
$$0\leq g_2^{(r)}(y)\leq c_5r^2.$$
This implies that for any $x\in {\mathbb R}^d$, $r\in [0,1/2]$ and $s\in [0,1)$, 
\begin{equation*}
\begin{split}
&\left[1-\frac{1}{(1+r^2+2s|x|r/{\langle x \rangle})^{\delta}}\right]
\left[\left(1+r^2-\frac{2s|x|}{{\langle x \rangle}}r\right)^q
-\left(1+r^2+\frac{2s|x|}{{\langle x \rangle}}r\right)^q
\right]
\leq c_5r^2.
\end{split}
\end{equation*}
By combining this with \eqref{eq:bdd0-1}, 
there exists $c_6>0$ such that for any $x\in {\mathbb R}^d$, $r\in [0,1/2]$ and $s\in [0,1)$, 
\begin{equation}\label{eq:j-3}
|J_q(\delta,x,r,s)|\leq c_6 r^2.
\end{equation}

Let $\delta_0=\eta+1/2\in (0,1/2)$. 
Then for any $\delta\in (0,\delta_0)$, we have $2(\eta-\delta)>-1$. 
Combining this with \eqref{eq:j-1}, \eqref{eq:j-2} and \eqref{eq:j-3}, 
we can apply the Lebesgue convergence theorem to show \eqref{eq:lim}.   
\end{proof}

\begin{proof}[Proof of Lemma {\rm \ref{lem:derivative} (2)}]
We compute the derivative of $\tilde{J}_q$ with respect to $\delta$:
\begin{equation}\label{eq:deri-j}
\begin{split}
&\frac{\partial \tilde{J}_q}{\partial \delta}(\delta, r,s)\\
&=-\frac{\log(r^2+2rs+1)}{(r^2+2rs+1)^{\delta}}[1+(r^2+2rs+1)^q]
-\frac{\log(r^2-2rs+1)}{(r^2-2rs+1)^{\delta}}[1+(r^2-2rs+1)^q].
\end{split}
\end{equation}
Since there exist $c_1>0$ and $c_2>0$ such that 
for any $r\geq 2$ and $s\in [0,1)$,
$$c_1r^2 \leq r^2\pm 2rs+1\leq c_2r^2, $$
there exists $c_3>0$, which is locally bounded in $\delta\in [0,\infty)$, 
such that 
\begin{equation}\label{eq:deri-1}
\left|\frac{\partial \tilde{J}_q}{\partial \delta}(\delta, r,s)\right|\leq c_3\frac{\log r}{r^{2\delta}}, 
\quad r\geq 2, \ s\in [0,1).
\end{equation}

There exists $c_4>0$ such that for any  $r\in [1/2,3/2]$ and $s\in [0,1)$,
\begin{equation}\label{eq:a}
\left|-\frac{\log(r^2+2rs+1)}{(r^2+2rs+1)^{\delta}}[1+(r^2+2rs+1)^q]\right|
\leq c_4.
\end{equation}
In a similar way to the proof of \eqref{eq:j-2}, 
there exists $c_5>0$ such that for any $r\in [1/2,3/2]$ and $s\in [0,1)$,
\begin{equation*}
\left|-\frac{\log(r^2-2rs+1)}{(r^2-2rs+1)^{\delta}}[1+(r^2-2rs+1)^q]\right|
\leq c_5|r-1|^{2(\eta-\delta)}\log(|r-1|^{-1})(1-s^2)^{q-\eta}.
\end{equation*}
Here $\eta$ is the same negative constant as in \eqref{eq:eta}. 
Combining this with \eqref{eq:a}, we obtain by \eqref{eq:deri-j},
\begin{equation}\label{eq:deri-2}
\left|\frac{\partial \tilde{J}_q}{\partial \delta}(\delta, r,s)\right|
\leq c_4+c_5|r-1|^{2(\eta-\delta)}\log(|r-1|^{-1})(1-s^2)^{q-\eta}, \quad r\in [1/2,3/2], \ s\in [0,1).
\end{equation}

For $r\in [0,1/2]$ and $s\in [0,1]$, 
we rewrite \eqref{eq:deri-j} as 
\begin{equation}\label{eq:deri-j-1}
\begin{split}
\frac{\partial \tilde{J}_q}{\partial \delta}(\delta, r,s)
&=-\frac{1}{(r^2+2rs+1)^{\delta}}\left[\log(r^2+2rs+1)+\log(r^2-2rs+1)\right]\\
&+\log(r^2-2rs+1)\left[\frac{1}{(r^2+2rs+1)^{\delta}}-\frac{1}{(r^2-2rs+1)^{\delta}}\right]\\
&-\frac{1}{(r^2+2rs+1)^{\delta-q}}\left[\log(r^2+2rs+1)+\log(r^2-2rs+1)\right]\\
&+\log(r^2-2rs+1)\left[\frac{1}{(r^2+2rs+1)^{\delta-q}}-\frac{1}{(r^2-2rs+1)^{\delta-q}}\right].
\end{split}
\end{equation}
Since $\log(1+x)/x\rightarrow 1$ as $x\rightarrow 0$ and   
\begin{equation*}
\log(r^2+2rs+1)+\log(r^2-2rs+1)
=2\log(1+r^2)+\log\left(1-\left(\frac{2rs}{1+r^2}\right)^2\right),
\end{equation*}
there exists $c_6>0$ such that for any $r\in [0,1/2]$ and $s\in [0,1)$,
\begin{equation}\label{eq:est-1}
|\log(r^2+2rs+1)+\log(r^2-2rs+1)|\leq c_6r^2.
\end{equation}
In the similar way, 
there exist $c_7>0$ and $c_8=c_8(\delta)>0$ such that 
for any $r\in [0,1/2]$ and $s\in [0,1)$,
\begin{equation}\label{eq:est-2}
\log(r^2+2rs+1)\leq c_7r
\end{equation}
and 
\begin{equation}\label{eq:est-3}
\left|\frac{1}{(r^2+2rs+1)^{\delta}}-\frac{1}{(r^2-2rs+1)^{\delta}}\right|
\leq c_8 r.
\end{equation}
Note that $c_8$ is locally bounded in $\delta\in [0,\infty)$ 
and the inequality above remains valid by replacing $\delta$ with $\delta-q$. 
Hence by \eqref{eq:deri-j-1}, \eqref{eq:est-1}, \eqref{eq:est-2} and \eqref{eq:est-3}, 
there exists $c_9=c_9(\delta)>0$,  which is locally bounded in $\delta\in [0,\infty)$, 
such that 
\begin{equation}\label{eq:deri-3}
\left|\frac{\partial \tilde{J}_q}{\partial \delta}(\delta, r,s)\right|
\leq c_9r^2, \quad r\in [0,1/2], \ s\in [0,1).
\end{equation}

Recall that $\delta_0=\eta+1/2\in (0,1/2)$. 
Then by \eqref{eq:eta}, $2(\eta-\delta)>-1$ for any $\delta\in [0,\delta_0)$ 
and $q-\eta+(d-3)/2>-1$. 
Therefore, by \eqref{eq:deri-1}, \eqref{eq:deri-2} and \eqref{eq:deri-3}, 
we apply the Lebesgue convergence theorem to get 
$$
\Lambda_q'(\delta)=\int_0^{\infty}
\left(\int_0^1 \frac{\partial \tilde{J}_q}{\partial \delta}(\delta, r,s)(1-s^2)^{(d-3)/2}\,{\rm d}s\right)
\frac{{\rm d}r}{r^{1+\alpha}}, \quad \delta\in [0,\delta_0).
$$
Since it follows by \eqref{eq:deri-j} that 
$$
\int_0^1 \frac{\partial \tilde{J}_q}{\partial \delta}(\delta, r,s)(1-s^2)^{(d-3)/2}\,{\rm d}s
=-\int_{-1}^1 K_q(\delta,r,s)(1-s^2)^{(d-3)/2}\,{\rm d}s,
$$
we arrive at \eqref{eq:lim}.
\end{proof}

\subsection{Proof of Lemma {\ref{lem:derivative}} (3)}
Let $\Gamma$ be the gamma function and 
$\psi=\Gamma'/\Gamma$ defined on ${\mathbb R}\setminus\{0,-1,-2,\dots\}$. 
Then by the relation $\Gamma(x+1)=x\Gamma(x)$, 
we have for any $n=1,2,3,\dots$,
\begin{equation}\label{eq:psi}
\psi(x+n)-\psi(x)
=\sum_{k=1}^n\frac{1}{x+k-1}, \quad x\in {\mathbb R}\setminus\{0,-1,-2,\dots\}.
\end{equation}
Let $B(x,y)$ be the beta function defined by 
$$B(x,y)=\frac{\Gamma(x)\Gamma(y)}{\Gamma(x+y)}.$$
Then 
$$B(x,y)=\int_0^1 t^{x-1}(1-t)^{y-1}\,{\rm d}t, \quad x>0, \ y>0.$$

\begin{proof}[Proof of Lemma \rm{\ref{lem:derivative} (3)}]
Let $q\in [(\alpha-d)/2,0)$. Then by the assertion (2), 
\begin{equation}\label{eq:lam-deri-0}
\begin{split}
-\Lambda'(0)
&=\int_0^{\infty}\left(\int_{-1}^1 \log(r^2+2rs+1)(1-s^2)^{(d-3)/2}\,{\rm d}s\right)\frac{{\rm d}r}{r^{\alpha+1}}\\
&+\int_0^{\infty}\left(\int_{-1}^1(r^2+2rs+1)^{(\alpha-d)/2}
\log(r^2+2rs+1)(1-s^2)^{(d-3)/2}\,{\rm d}s\right)\,\frac{{\rm d}r}{r^{\alpha+1}}\\
&=I_1+I_2.
\end{split}
\end{equation}

We first compute $I_1$. 
By the Taylor expansion,
\begin{equation}\label{eq:series}
\begin{split}
\log(r^2+2rs+1)
&=\log(1+r^2)+\log\left(1+\frac{2rs}{1+r^2}\right)\\
&=\log(1+r^2)+\sum_{n=1}^{\infty}\frac{(-1)^{n-1}}{n}\left(\frac{2rs}{1+r^2}\right)^n.
\end{split}
\end{equation}
Since 
$$\int_{-1}^1 s^{2n+1}(1-s^2)^{(d-3)/2}\,{\rm d}s=0$$
and 
$$\int_{-1}^1 s^{2n}(1-s^2)^{(d-3)/2}\,{\rm d}s
=2\int_0^1s^{2n}(1-s^2)^{(d-3)/2}\,{\rm d}s
=B\left(n+\frac{1}{2}, \frac{d-1}{2}\right),$$
we have by \eqref{eq:series},
\begin{equation}\label{eq:int-exp}
\begin{split}
&\int_{-1}^1 \log(r^2+2rs+1)(1-s^2)^{(d-3)/2}\,{\rm d}s\\
&=B\left(\frac{1}{2}, \frac{d-1}{2}\right)\log(1+r^2)
-\sum_{n=1}^{\infty}\frac{2^{2n}}{2n}B\left(n+\frac{1}{2}, \frac{d-1}{2}\right)\left(\frac{r}{1+r^2}\right)^{2n}.
\end{split}
\end{equation}

By the integration by parts and change of variables formulae,
$$\int_0^{\infty}\frac{\log(1+r^2)}{r^{\alpha+1}}\,{\rm d}r
=\frac{1}{\alpha}B\left(\frac{\alpha}{2},1-\frac{\alpha}{2}\right)$$
and 
$$\int_0^{\infty}\frac{1}{r^{\alpha+1}}\left(\frac{r}{1+r^2}\right)^{2n}\,{\rm d}r
=\frac{1}{2}B\left(n+\frac{\alpha}{2},n-\frac{\alpha}{2}\right) \quad (n=1,2,3,\dots).
$$
Hence by \eqref{eq:int-exp}, 
\begin{equation}\label{eq:int-series}
I_1=\frac{1}{\alpha}B\left(\frac{1}{2}, \frac{d-1}{2}\right)B\left(\frac{\alpha}{2}, 1-\frac{\alpha}{2}\right)
-\sum_{n=1}^{\infty}\frac{2^{2n}}{4n}
B\left(n+\frac{1}{2}, \frac{d-1}{2}\right)B\left(n+\frac{\alpha}{2}, n-\frac{\alpha}{2}\right).
\end{equation}

For $\lambda\in {\mathbb R}$, 
let $(\lambda)_0=1$ and 
$(\lambda)_n=\lambda(\lambda+1)\cdots(\lambda+n-1) \ (n\in {\mathbb N})$. 
Since
$\Gamma(n+\lambda)=(\lambda)_n\Gamma(\lambda)$, we have 
\begin{equation*}
\begin{split}
&B\left(n+\frac{1}{2},\frac{d-1}{2}\right)B\left(n+\frac{\alpha}{2},n-\frac{\alpha}{2}\right)\\
&=\frac{\Gamma(n+1/2)\Gamma((d-1)/2)}{\Gamma(n+d/2)}
\frac{\Gamma(n+\alpha/2)\Gamma(n-\alpha/2)}{\Gamma(2n)}\\
&=-\frac{2}{\alpha}\frac{\Gamma(1/2)\Gamma((d-1)/2)}{\Gamma(d/2)}
\frac{(1/2)_n}{(d/2)_n}
\frac{\Gamma(\alpha/2)\Gamma(1-\alpha/2)}{\Gamma(2n)}({\alpha/2})_n({-\alpha/2})_n\\
&=-\frac{2}{\alpha}B\left(\frac{1}{2}, \frac{d-1}{2}\right)B\left(\frac{\alpha}{2}, 1-\frac{\alpha}{2}\right)
\frac{(2n-1)!!}{2^{n}(2n-1)!}\frac{({\alpha/2})_n({-\alpha/2})_n}{(d/2)_n}
\end{split}
\end{equation*}
so that
\begin{equation}\label{eq:beta}
\begin{split}
&\frac{2^{2n}}{n}B\left(n+\frac{1}{2},\frac{d-1}{2}\right)B\left(n+\frac{\alpha}{2},n-\frac{\alpha}{2}\right)\\
&=-\frac{4}{\alpha}B\left(\frac{1}{2}, \frac{d-1}{2}\right)B\left(\frac{\alpha}{2}, 1-\frac{\alpha}{2}\right)
\frac{1}{n!}\frac{({\alpha/2})_n({-\alpha/2})_n}{(d/2)_n}.
\end{split}
\end{equation}
Since we know by \cite[p.174, (7.4.14)]{H75} that 
$$\sum_{n=0}^{\infty}\frac{1}{n!}\frac{({\alpha/2})_n({-\alpha/2})_n}{(d/2)_n}
=\frac{\Gamma(d/2)^2}{\Gamma((d+\alpha)/2)\Gamma((d-\alpha)/2)},$$
it follows by \eqref{eq:int-series} and \eqref{eq:beta} that 
\begin{equation}\label{eq:exp-1}
\begin{split}
I_1
&=\frac{1}{\alpha}B\left(\frac{1}{2},\frac{d-1}{2}\right)B\left(\frac{\alpha}{2},1-\frac{\alpha}{2}\right)
\sum_{n=0}^{\infty}\frac{1}{n!}\frac{({\alpha/2})_n({-\alpha/2})_n}{(d/2)_n}\\
&=-
\frac{\Gamma(1/2)\Gamma(d/2)\Gamma((d-1)/2)\Gamma(1+\alpha/2)\Gamma(-\alpha/2)}
{\alpha\Gamma((d+\alpha)/2)\Gamma((d-\alpha)/2)}.
\end{split}
\end{equation}

We next compute $I_2$.
By the Taylor expansion again,
\begin{equation*}
\begin{split}
&(r^2+2rs+1)^q\log(r^2+2rs+1)\\
&=(1+r^2)^q\left[\sum_{n=0}^{\infty}\binom{q}{n}\left(\frac{2rs}{1+r^2}\right)^n\right]
\left[\log(1+r^2)+\sum_{n=1}^{\infty}\frac{(-1)^{n-1}}{n}\left(\frac{2rs}{1+r^2}\right)^n\right]\\
&=(1+r^2)^q\log(1+r^2) \sum_{n=0}^{\infty}\binom{q}{n}\left(\frac{2r}{1+r^2}\right)^ns^n\\
&\qquad -(1+r^2)^q\sum_{n=1}^{\infty}
\left[\sum_{k=0}^{n-1}\binom{q}{k}\frac{(-1)^{n-k}}{n-k}\right]\left(\frac{2r}{1+r^2}\right)^n s^n.
\end{split}
\end{equation*}
Then
\begin{equation}\label{eq:pro-1}
\begin{split}
&\int_{-1}^1(r^2+2rs+1)^q\log(r^2+2rs+1)(1-s^2)^{(d-3)/2}\,{\rm d}s\\
&=(1+r^2)^q\log(1+r^2) \sum_{n=0}^{\infty}\binom{q}{2n}B\left(n+\frac{1}{2},\frac{d-1}{2}\right)
\left(\frac{2r}{1+r^2}\right)^{2n}\\
&\qquad -(1+r^2)^q\sum_{n=1}^{\infty}
\left[\sum_{k=0}^{2n-1}\binom{q}{k}\frac{(-1)^k}{2n-k}\right]B\left(n+\frac{1}{2},\frac{d-1}{2}\right)
\left(\frac{2r}{1+r^2}\right)^{2n}.
\end{split}
\end{equation}
Since
$$\int_0^{\infty}(1+r^2)^q\left(\frac{r}{1+r^2}\right)^{2n}\frac{{\rm d}r}{r^{\alpha+1}}
=\frac{1}{2}B\left(n-\frac{\alpha}{2}, n+\frac{\alpha}{2}-q\right), \quad n=1,2,3,\dots$$
and 
\begin{equation*}
\begin{split}
&\int_0^{\infty}(1+r^2)^q\log(1+r^2)\left(\frac{r}{1+r^2}\right)^{2n}\frac{{\rm d}r}{r^{\alpha+1}}\\
&=\frac{1}{2}
\left(\psi\left(n+\frac{\alpha}{2}-q\right)-\psi(2n-q)\right)
B\left(n-\frac{\alpha}{2}, n+\frac{\alpha}{2}-q\right), \quad n=0,1,2,\dots
\end{split}
\end{equation*}
by using \cite[p.39, (4.41)]{O74},
we obtain by \eqref{eq:pro-1},
\begin{equation}\label{eq:int-psi}
\begin{split}
&\int_0^{\infty}\left(\int_{-1}^1(r^2+2rs+1)^q\log(r^2+2rs+1)(1-s^2)^{(d-3)/2}\,{\rm d}s\right)\,\frac{{\rm d}r}{r^{\alpha+1}}\\
&=\frac{1}{2}\sum_{n=0}^{\infty}\binom{q}{2n}2^{2n}B\left(n+\frac{1}{2},\frac{d-1}{2}\right)
B\left(n-\frac{\alpha}{2},n+\frac{\alpha}{2}-q\right)
\left[\psi(2n-q)-\psi\left(n+\frac{\alpha}{2}-q\right)\right]\\
&\qquad -\frac{1}{2}\sum_{n=1}^{\infty}
\left[\sum_{k=0}^{2n-1}\binom{q}{k}\frac{(-1)^k}{2n-k}\right]
2^{2n}B\left(n+\frac{1}{2},\frac{d-1}{2}\right)
B\left(n-\frac{\alpha}{2},n+\frac{\alpha}{2}-q\right).
\end{split}
\end{equation}

If $a$ is not an integer and $b-a\ne -1$, then we have by induction, 
$$\sum_{k=0}^n\frac{(b)_k}{(a)_k}=\frac{1}{b-a+1}\left(1-a+\frac{b\cdot(b+1)_n}{(a)_n}\right), 
\quad n=0,1,2,\dots.$$
Using this equality, we also see by induction and \eqref{eq:psi} that 
if $a$ is not a negative integer, then 
$$\sum_{k=1}^n\frac{(-n)_k}{k(a-n+1)_k}=-\sum_{k=1}^n\frac{1}{a-n+k}
=\sum_{l=1}^n\frac{1}{-a+l-1}=\psi(n-a)-\psi(-a).$$
Since this and \eqref{eq:psi} yield 
\begin{equation*}
\begin{split}
\sum_{k=0}^{n-1}\binom{a}{k}\frac{(-1)^k}{n-k}
=(-1)^n\binom{a}{n}\sum_{k=1}^n\frac{(-n)_k}{k(a-n+1)_{k}}
&=(-1)^n\binom{a}{n}(\psi(n-a)-\psi(-a)),
\end{split}
\end{equation*}
the last expression of \eqref{eq:int-psi} is equal to 
\begin{equation}\label{eq:q-psi}
\frac{1}{2}\sum_{n=0}^{\infty}\binom{q}{2n}2^{2n}
B\left(n+\frac{1}{2},\frac{d-1}{2}\right)
B\left(n-\frac{\alpha}{2},n+\frac{\alpha}{2}-q\right)
\left[\psi(-q)-\psi\left(n+\frac{\alpha}{2}-q\right)\right].
\end{equation}

Note that 
$$\binom{(\alpha-d)/2}{2n}2^{2n}
B\left(n+\frac{1}{2},\frac{d-1}{2}\right)
B\left(n-\frac{\alpha}{2},n+\frac{d}{2}\right)
=\frac{\Gamma(1/2)\Gamma((d-1)/2)\Gamma(n-\alpha/2)}{n!\Gamma((d-\alpha)/2)}.$$
We also see that 
$$\sum_{n=0}^{\infty}\frac{\Gamma(n-\alpha/2)}{n!}
=\Gamma\left(-\frac{\alpha}{2}\right)\sum_{n=0}^{\infty}\binom{\alpha/2}{n}(-1)^n=0,$$
and by \cite[p.361 (55.4.1) and p.151 (7.1.2)]{H75},
\begin{equation*}
\begin{split}
\sum_{n=0}^{\infty}\frac{\Gamma(n-\alpha/2)}{n!}\psi\left(n+\frac{d}{2}\right)
&=\Gamma\left(-\frac{\alpha}{2}\right)\sum_{n=0}^{\infty}\frac{(-\alpha/2)_n}{n!}\psi\left(n+\frac{d}{2}\right)\\
&=-\frac{\Gamma(-\alpha/2)\Gamma(1+\alpha/2)\Gamma(d/2)}{\Gamma(1+(d+\alpha)/2)}
\sum_{n=0}^{\infty}\frac{(d/2)_n}{(1+(d+\alpha)/2)_n}\\
&=-\frac{2\Gamma(-\alpha/2)\Gamma(1+\alpha/2)\Gamma(d/2)}{\alpha\Gamma((d+\alpha)/2)}. 
\end{split}
\end{equation*}
Hence if we take $q=(\alpha-d)/2$ in \eqref{eq:int-psi} and \eqref{eq:q-psi}, 
then by \eqref{eq:exp-1},
\begin{equation*}
\begin{split}
I_2&=-\frac{\Gamma(1/2)\Gamma((d-1)/2)}{2\Gamma((d-\alpha)/2)}
\sum_{n=0}^{\infty}\frac{\Gamma(n-\alpha/2)}{n!}
\psi\left(n+\frac{d}{2}\right)\\
&=\frac{\Gamma(1/2)\Gamma(d/2)\Gamma((d-1)/2)\Gamma(1+\alpha/2)\Gamma(-\alpha/2)}
{\alpha\Gamma((d+\alpha)/2)\Gamma((d-\alpha)/2)}=-I_1.
\end{split}
\end{equation*}
By this equality and \eqref{eq:lam-deri-0}, the proof is complete.
\end{proof}

\appendix
\section{Recurrence}\label{sect:rec}
In this appendix, we present a version of the recurrence criterion 
for non-local Dirichlet forms 
established in \cite[Theorem 4]{O95} and \cite[Theorem 2.4]{OU15}. 

\subsection{Recurrence criterion}
We first recall a necessary and sufficient condition for recurrence of regular Dirichlet forms 
in terms of the capacity (\cite{O95, O96}).
Let $X$ be a locally compact separable metric space and 
$m$ a positive Radon measure on $X$ with full support. 
Let $X_{\Delta}=X\cup\{\Delta\}$ denote the one point compactification of $X$. 
Let $({\cal E},{\cal F})$ be a regular Dirichlet form on $L^2(X;m)$. 
We say that a function $u$ on $X$ belongs locally to ${\cal F}$ 
if for any relatively compact set $G$ in $X$, 
there exists $u_G\in {\cal F}$ such that $u=u_G$, $m$-a.e.\ on $G$.
Let ${\cal F}_{{\rm loc}}$ denote the totality of such functions. 
In what follows, we impose the next assumption on $({\cal E},{\cal F})$:
\begin{assum}\label{assum:exh}\rm 
There exists a nonnegative function $\rho$ on $X$ such that 
\begin{enumerate}
\item[(i)]  $\rho\in {\cal F}_{{\rm loc}}$ and $\rho(x)\rightarrow \infty$ as $x\rightarrow\Delta$;
\item[(ii)] For any $r>0$, the set $B_{\rho}(r):=\{y\in X \mid \rho(y)<r\}$ is relatively compact in $X$.
\end{enumerate}
\end{assum}

Fix a function $\rho$ satisfying  Assumption \ref{assum:exh}. 
For $r,R>0 \ (0<r<R)$, we define 
$${\mathbf C}(r,R)=\inf\left\{{\cal E}(u,u) \mid 
u\in {\cal F}\cap C_0(X), \ \text{$u\geq 1$ on $\overline{B_{\rho}(r)}$ and  $u=0$ on $B_{\rho}(R)^c$} \right\}.$$

\begin{thm}\label{thm:capa-ineq}{\rm (\cite[Theorem 1.1]{O95} or \cite[Theorem 1]{O96})} 
Under Assumption {\rm \ref{assum:exh}}, 
$({\cal E},{\cal F})$ is recurrent if and only if 
for each fixed  $r>0$, 
$$\lim_{R\rightarrow\infty}{\mathbf C}(r,R)=0.$$ 
\end{thm}

Using this theorem, we show a recurrence criterion for regular Dirichlet forms. 
Let ${\rm diag}=\{(x,y)\in X\times X \mid x=y\}$. 
We now suppose that $({\cal E},{\cal F})$ has no killing measure 
in the Beurling-Deny representation 
(\cite[Theorem 3.2.1 and Lemma 4.5.4]{FOT11}); that is, 
there exist a symmetric form $({\cal E}^{(c)},{\cal F}\cap C_0(X))$ 
with the strongly local property (see \cite[p.120]{FOT11} for definition) and  
a symmetric positive Radon measure $J({\rm d}x\,{\rm d}y)$ on $X\times X\setminus {\rm diag}$ such that 
$${\cal E}(u,u)={\cal E}^{(c)}(u,u)
+\iint_{X\times X\setminus {\rm diag}}(u(x)-u(y))^2\,J({\rm d}x\,{\rm d}y), \quad u\in {\cal F}\cap C_0(X).$$
We can then extend ${\cal E}^{(c)}$ uniquely to ${\cal F}$.  
Moreover, for each $u\in {\cal F}$, 
there exists a positive Radon measure $\mu_{\langle u\rangle}^{c}$ on $X$ 
such that ${\cal E}^{(c)}(u,u)=\mu_{\langle u \rangle}^{c}(X)/2$ (\cite[p.123]{FOT11}).

Let $B(r)=B_{\rho}(r)$. 
We then define 
$$M_0(r)=\mu_{\langle \rho\rangle}^c(B(r)), \quad 
M_1(r)=\iint_{B(r)\times B(r)}(\rho(x)-\rho(y))^2\,J({\rm d}x\,{\rm d}y), \quad r>0$$
and 
$$M_2(r,R)=\iint_{B(r)\times B(R)^c}\log\left(\frac{\rho(y)}{R}\right)\,J({\rm d}x\,{\rm d}y), 
\quad 0<r<R<\infty.$$
\begin{thm}\label{thm:test-r}
Assume that there exist $c_0>1$, $r_0>1$, and 
a positive continuous and  nondecreasing function $L$ on $[r_0,\infty)$, 
which is slowly varying at infinity, 
such that for any $r\geq r_0$,
\begin{equation}\label{eq:upper-varying}
r^{-2}(M_0(r)+M_1(r))+M_2(r,c_0r)\leq L(r).
\end{equation}
Then there exists $K>0$ such that for any $r,R>0$ with $r_0\leq r<2c_0r\leq R$,
\begin{equation}\label{eq:capa-u}
{\mathbf C}(r,R)\leq K\left(\int_r^R\frac{{\rm d}s}{sL(s)}\right)^{-1}.
\end{equation}
In particular, $({\cal E},{\cal F})$ is recurrent if 
\begin{equation}\label{eq:test}
\int_{r_0}^{\infty}\frac{{\rm d}s}{sL(s)}=\infty.
\end{equation}
\end{thm}

Theorem \ref{thm:test-r} is a version of \cite[Theorem 4]{O95}  
and \cite[Theorem 2.4]{OU15};
these theorems require an upper bound of the form 
$J(B(r)\times B(R)^c)\leq F(r)G(R)$ for any large $r,R\geq 1$ 
satisfying  $R/r\geq c_0$ with some $c_0>1$. 
Theorem \ref{thm:test-r} says that 
we can replace the upper bound condition of $J(B(r)\times B(R)^c)$ 
by that of $M_2(r,c_0r)$.

We now show Theorem \ref{thm:test-r} by following the proof of \cite[Theorem 4]{O96}. 

\begin{proof}[Proof of Theorem {\rm \ref{thm:test-r}}]
Let $r_0>1$  and $c_0>1$ be the same constants as in the statement of the theorem. 
Fix $r,R>0$ with $r_0\leq r<2c_0r\leq R$. 
Then there exists a positive integer $N$ such that  $(2c_0)^N\leq R/r<(2c_0)^{N+1}$, 
which yields $R/r=c^N$ for some $c\in [2c_0,(2c_0)^2)$.

For any constants $s>0$, $\kappa>1$ and positive integer $n$, 
$$J(B(s)\times B(\kappa^n s)^c)
\leq \frac{\kappa}{\kappa-1}
\int_{\kappa^{n-1}}^{\kappa^n} J(B(s)\times B(ts)^c)\frac{{\rm d}t}{t}.$$
Hence if we let $s_n=c^nr \ (n=0,\dots, N)$, then  
\begin{equation*}
\begin{split}
&\sum_{l=n+2}^{N-1}J(B(s_n)\times B(s_l)^c)
=\sum_{l=n+2}^{N-1}J(B(s_n)\times B(c^{l-n}s_n)^c)\\
&\leq \frac{c}{c-1}\sum_{l=n+2}^{N-1}\int_{c^{l-n-1}}^{c^{l-n}}J(B(s_n)\times B(ts_n)^c)\frac{{\rm d}t}{t}
\leq 2\int_c^{\infty}J(B(s_n)\times B(ts_n)^c)\frac{{\rm d}t}{t}.
\end{split}
\end{equation*}
In the same way, we also have 
$$J(B(s_n)\times B(cs_n)^c)
\leq \frac{c}{c-c_0}\int_{c_0}^cJ(B(s_n)\times B(ts_n)^c)\frac{{\rm d}t}{t}
\leq 2\int_{c_0}^cJ(B(s_n)\times B(ts_n)^c)\frac{{\rm d}t}{t}$$
so that 
$$\sum_{l=n+1}^{N-1}J(B(s_n)\times B(s_l)^c)
\leq 2\int_{c_0}^{\infty}J(B(s_n)\times B(ts_n)^c)\frac{{\rm d}t}{t}.$$
Since the Fubini theorem implies that
\begin{equation*}
\begin{split}
&\int_{c_0}^{\infty}J(B(s_n)\times B(ts_n)^c)\frac{{\rm d}t}{t}
=\int_{c_0}^{\infty}
\left(\iint_{B(s_n)\times B(ts_n)^c}J({\rm d}x\,{\rm d}y)\right)\frac{{\rm d}t}{t}\\
&=\iint_{X \times X} {\bf 1}_{B(s_n)}(x)
\left(\int_{c_0}^{\infty}{\bf 1}_{B(ts_n)^c}(y)\frac{{\rm d}t}{t}\right)\,J({\rm d}x\,{\rm d}y)\\
&=\iint_{B(s_n)\times B(c_0s_n)^c} 
\log\left(\frac{\rho(y)}{c_0s_n}\right)\,J({\rm d}x\,{\rm d}y)=M_2(s_n,c_0s_n),
\end{split}
\end{equation*}
we get
\begin{equation}\label{eq:bound-log-0}
\sum_{l=n+1}^{N-1}J(B(s_n)\times B(s_l)^c)
\leq 2M_2(s_n,c_0s_n).
\end{equation}

Let
$$
\varphi_n(x)=0\vee \left(\frac{s_n-\rho(x)}{s_n-s_{n-1}}\right)\wedge 1.
$$
Then 
$${\cal E}^{(c)}(\varphi_n,\varphi_n)
=\mu_{\langle\varphi_n\rangle}^c(B(s_n))
\leq \frac{1}{(s_n-s_{n-1})^2}M_0(s_n)$$ 
and for any $k,l\in \{1,\dots, N\}$ with $k\ne l$, 
we have ${\cal E}^{(c)}(\varphi_n,\varphi_l)=0$ 
by the strong local property (\cite[p.120]{FOT11}). 
Hence by following the proof of \cite[Theorem 4]{O96} 
and using \eqref{eq:bound-log-0}, 
there exists $c_1>0$ such that  
\begin{equation}\label{eq:e-bound}
\begin{split}
&{\cal E}(\varphi_n,\varphi_n)
+2\sum_{l=n+1}^{N}|{\cal E}(\varphi_n,\varphi_l)|\\
&\leq \frac{1}{(s_n-s_{n-1})^2}M_0(B(s_n))\\
&+\frac{1}{s_n-s_{n-1}}\left(\frac{1}{s_n-s_{n-1}}+\frac{2}{s_{n+1}-s_n}\right)
M_1(s_{n+1})
+10\sum_{l=n+1}^{N-1}J(B(s_n)\times B(s_l)^c)\\
&\leq c_1\left\{\frac{1}{s_n^2}(M_0(s_n)+M_1(s_{n+1}))+M_2(s_n,c_0s_n)\right\}.
\end{split}
\end{equation}
Here we recall that the function $L(r)$ is continuous and slowly varying at infinity. 
Therefore, by the uniform convergence theorem (\cite[Theorem 1.2.1]{BGT89}) 
or \cite[Lemma 5.4]{OU15},   
there exist $C_1>0$ and $C_2>0$ 
for any positive constants $a$, $b$ with $0<a<b$, such that 
$$C_1\leq \frac{L(\lambda r)}{L(r)}\leq C_2, \quad r\geq 1, \ \lambda\in [a,b].$$
Furthermore, since $s_{n+1}/s_{n-1}=c^2$ and $s_n/s_{n-1}=c$ for any $n\geq 1$, 
we see by \eqref{eq:upper-varying} that for some $c_2>0$,
\begin{equation}\label{eq:m-bound}
\frac{1}{s_n^2}(M_0(s_n)+M_1(s_{n+1}))+M_2(s_n,c_0s_n)\leq c_2L(s_{n-1}), \quad n=1,\dots, N.
\end{equation}

Let 
$$
C=\left(\sum_{n=1}^N\frac{1}{L(s_{n-1})}\right)^{-1}, \quad p_n=\frac{C}{L(s_{n-1})}
$$
and 
$$
u_N=\sum_{n=1}^N p_n \varphi_n.
$$
Then \cite[Theorem 1.5.4]{BGT89} shows that  for any $\alpha>0$, 
there exists a nonincreasing  function $\phi$ on $[r_0,\infty)$ such that
$\lim_{r\rightarrow\infty}L(r)/(r^{\alpha}\phi(r))=1$. 
Therefore,   
there exists $c_{\alpha}>0$ such that for any $\kappa>1$, $s\geq r_0$ and $t\geq r_0$ with 
$s\leq t\leq \kappa s$, 
$$L(s)=s^{\alpha}\frac{L(s)}{s^{\alpha}}
\geq c_{\alpha}\left(\frac{t}{\kappa}\right)^{\alpha}\frac{L(t)}{t^{\alpha}}
=\frac{c_{\alpha}}{\kappa^{\alpha}}L(t).$$
Since $L(s)$ is nondecreasing by assumption, 
we have by \eqref{eq:e-bound} and \eqref{eq:m-bound},  
\begin{equation*}
\begin{split}
{\cal E}(u_N,u_N)
&=\sum_{n=1}^N p_n^2{\cal E}(\varphi_n,\varphi_n)
+2\sum_{n=1}^{N-1}p_n\sum_{l=n+1}^{N}p_l{\cal E}(\varphi_n,\varphi_l)\\
&\leq \sum_{n=1}^N p_n^2\left({\cal E}(\varphi_n,\varphi_n)
+2\sum_{l=n+1}^{N}|{\cal E}(\varphi_n,\varphi_l)|\right)\\
&\leq c_3\sum_{n=1}^N\frac{C^2}{L(s_{n-1})}
=c_3 C\leq c_3\left(\frac{1}{c-1}\int_r^R\frac{{\rm d}s}{sL(s)}\right)^{-1}.
\end{split}
\end{equation*}
As ${\mathbf C}(r,R)\leq {\cal E}(u_N,u_N)$, 
we obtain \eqref{eq:capa-u}. 
In particular, if \eqref{eq:test} holds, then $({\cal E},{\cal F})$ is recurrent  
by \eqref{eq:capa-u} and Theorem \ref{thm:capa-ineq}.
\end{proof}

\subsection{Application}
In this subsection, we apply Theorem \ref{thm:test-r} to 
regular Dirichlet forms on $L^2({\mathbb R}^d)$ 
with unbounded/degenerate coefficients. 
Let $({\cal E},{\cal F})$ be a regular Dirichlet form on $L^2({\mathbb R}^d)$ as in Subsection \ref{sub:set} 
with Assumption \ref{assum:lower} replaced by the next assumption:
\begin{assum}\label{assum:upper}
There exists a kernel $J_0(x,{\rm d}y)$ on ${\mathbb R}^d\times {\cal B}({\mathbb R}^d)$ 
satisfying the following{\rm :}
\begin{enumerate}
\item[{\rm (i)}] Assumption {\rm \ref{assum:lower}} holds with the inequality \eqref{eq:jump-lower} 
replaced by the reversed one{\rm ;} for any $x\in {\mathbb R}^d$ and $A\in {\cal B}({\mathbb R})$, 
\begin{equation}\label{eq:jump-upper}
J_0(x,x+A)\leq \nu(A).
\end{equation}
\item[{\rm (ii)}] There exists a positive  Borel measurable function $f$ on $(0,\infty)$ 
such that 
$${\bf 1}_{\{|y|\geq 1\}}\nu({\rm d}y)=f(|y|)\,{\rm d}y.$$
Moreover, there exist positive constants $c_1$, $c_2$, $c_3$, $c_4$ such that   
for any positive constants $a,b$ with $c_1\leq b/a\leq c_2$, then $c_3\leq f(b)/f(a)\leq c_4$.
\end{enumerate}
\end{assum}

Under Assumption \ref{assum:upper}, 
Theorem \ref{thm:test-r} is applicable to $({\cal E},{\cal F})$ because 
it fulfills Assumption \ref{assum:exh} with $\rho(x)=|x|$.

Let $B(r)=\{x\in {\mathbb R}^d \mid |x|<r\}$. 
By the symmetry of the measure $J_0(x,{\rm d}y){\rm d}x$ and \eqref{eq:jump-upper}, 
we have
\begin{equation}\label{eq:s-jump}
\begin{split}
&\iint_{B(s)\times B(s)}|x-y|^2{\bf 1}_{\{|x-y|<1\}}\,J({\rm d}x\,{\rm d}y)\\
&=2\iint_{B(s)\times B(s)}a_1(x)|x-y|^2{\bf 1}_{\{|x-y|<1\}}\,J_0(x,{\rm d}y){\rm d}x\\
&\leq 2\int_{B(s)}a_1(x)\,{\rm d}x\int_{0<|z|<1}|z|^2\,\nu({\rm d}z)
\end{split}
\end{equation}
and
\begin{equation*}
\begin{split}
\iint_{B(s)\times B(s)}|x-y|^2{\bf 1}_{\{|x-y|\geq 1\}}\,J({\rm d}x\,{\rm d}y)
\leq 2\int_{B(s)}a_2(x)\,{\rm d}x\int_{1\leq |z|<2s}|z|^2\,\nu({\rm d}z).
\end{split}
\end{equation*}
Hence 
\begin{equation*}
\begin{split}
&\iint_{B(s)\times B(s)}|x-y|^2\,J({\rm d}x\,{\rm d}y)\\
&\leq 2\left(\int_{B(s)}a_1(x)\,{\rm d}x\int_{0<|z|<1}|z|^2\,\nu({\rm d}z)
+\int_{B(s)}a_2(x)\,{\rm d}x\int_{1\leq |z|<2s}|z|^2\,\nu({\rm d}z)\right).
\end{split}
\end{equation*}

Let $c_0>2$ and $r_0>1$. 
Since $(c_0-1)r_0>1$, we have for any $s\geq r_0$,
$$\iint_{B(s)\times B(c_0s)^c}\log\left(\frac{|y|}{c_0s}\right){\bf 1}_{\{|x-y|<1\}}\,J({\rm d}x \, {\rm d}y)=0.$$
Note that for any $x\in B(s)$ and $y\in B(c_0s)^c$, 
we have $|x-y|/2\leq |y|\leq 2|x-y|$, 
which implies that $f(y)/f(y-x)$ is bounded from below and above by positive constants. 
Then, in a similar way to \eqref{eq:s-jump}, we also obtain
\begin{equation*}
\begin{split}
&\iint_{B(s)\times B(c_0s)^c}\log\left(\frac{|y|}{c_0s}\right){\bf 1}_{\{|x-y|\geq 1\}}\,J({\rm d}x \, {\rm d}y)\\
&\leq \int_{B(s)}a_2(x)\,{\rm d}x
\int_{|y|\geq (c_0-1)s}\log\left(\frac{|y|}{s}\right)\,\nu({\rm d}y)
+\omega_d s^d
\int_{|y|\geq (c_0-1)s}a_2(y)\log\left(\frac{|y|}{s}\right)\,\nu({\rm d}y).
\end{split}
\end{equation*}
In particular,  if we let 
\begin{equation}\label{eq:n(s)}
\begin{split}
N(s)
&=\frac{1}{s^2}\int_{B(s)}a_1(x)\,{\rm d}x\\
&+\int_{B(s)}a_2(x)\,{\rm d}x
\left[\int_{1\leq |y|<s}\left(\frac{|y|}{s}\right)^2 \,\nu({\rm d}y)
+\int_{|y|\geq s}\log\left(\frac{|y|}{s}\right)\,\nu({\rm d}y)\right]\\
&+s^d\int_{|y|\geq s}a_2(y)\log\left(\frac{|y|}{s}\right)\,\nu({\rm d}y),
\end{split}
\end{equation}
then by Theorem \ref{thm:test-r}, 
$({\cal E},{\cal F})$ is recurrent if $\int_{\cdot}^{\infty}(sN(s))^{-1}\,{\rm d}s=\infty$.

\begin{exam}\label{exam:rec}\rm 
We show a recurrence criterion in terms of the coefficient growth at infinity.  
Even though the first example below is already given in \cite[Example 4.4]{OU15} 
under the general setting, 
we state it for the sake of completeness. 
\begin{enumerate}
\item[(1)] 
Let $\alpha\in (0,2)$ and $\beta\in (0,\infty)$. 
Let $J(x,{\rm d}y)$ satisfy Assumption \ref{assum:upper} with 
$$\nu({\rm d}z)
={\bf 1}_{\{0<|z|<1\}}\frac{{\rm d}z}{|z|^{d+\alpha}}
+{\bf 1}_{\{|z|\geq 1\}}\frac{{\rm d}z}{|z|^{d+\beta}}.$$
Assume first that $\beta\ne 2$. Then for some $c>0$,
$$N(s)\leq 
c\left(\frac{1}{s^2}\int_{B(s)}a_1(x)\,{\rm d}x+\frac{1}{s^{\beta\wedge 2}}\int_{B(s)}a_2(x)\,{\rm d}x
+s^d\int_{|y|\geq s}\frac{a_2(y)}{|y|^{d+\beta}}\log\left(\frac{|y|}{s}\right)\,{\rm d}y\right).$$
Hence if there exist $c_1>0$ and $c_2>0$ such that 
for any $x\in {\mathbb R}^d$,
\begin{equation}\label{eq:bound-coeff}
a_1(x)\leq c_1(1+|x|)^{2-d}\log(2+|x|), \quad a_2(x)\leq c_2(1+|x|)^{(\beta\wedge 2)-d}\log(2+|x|),
\end{equation}
then there exists $c_3>0$ such that 
$N(s)\leq c_3\log(2+s)$ for any $s\geq 2$. 
Since $\int_2^{\infty}(sN(s))^{-1}\,{\rm d}s=\infty$, 
$({\cal E},{\cal F})$ is recurrent.

Assume next that $\beta=2$. Then 
$$N(s)\leq 
c\left(\frac{1}{s^2}\int_{B(s)}a_1(x)\,{\rm d}x+\frac{\log s}{s^2}\int_{B(s)}a_2(x)\,{\rm d}x
+s^d\int_{|y|\geq s}\frac{a_2(y)}{|y|^{d+2}}\log\left(\frac{|y|}{s}\right)\,{\rm d}y\right).$$
In particular, if there exist $c_1>0$ and $c_2>0$ such that 
for any $x\in {\mathbb R}^d$,
$$
a_1(x)\leq c_1(1+|x|)^{2-d}\log(2+|x|), \quad a_2(x)\leq c_2(1+|x|)^{2-d}\log(\log(3+|x|)),
$$
then there exists $c_3>0$ such that 
$N(s)\leq c_3\log(2+s)\log\log(2+s)$ for any $s\geq 2$. 
Since $\int_2^{\infty}(sN(s))^{-1}\,{\rm d}s=\infty$,  $({\cal E},{\cal F})$ is recurrent. 

We note that a similar result is valid for $\nu({\rm d}z)
={\bf 1}_{\{0<|z|<1\}}|z|^{-(d+\alpha)}\,{\rm d}z$
or $\nu({\rm d}z)={\bf 1}_{\{|z|\geq 1\}}|z|^{-(d+\beta)}\,{\rm d}z$.

\item[(2)] 
Let $\alpha\in (0,1)$ and $\beta\in (0,\infty)$. 
If  
$$\nu({\rm d}z)
={\bf 1}_{\{0<|z|<1\}}\frac{{\rm d}z}{|z|^{d+2}(\log (2+|z|))^{\alpha}}
+{\bf 1}_{\{|z|\ge 1\}}\frac{{\rm d}z}{|z|^d(\log(2+|z|))^{\beta+2}},$$
then for some $c>0$,
\begin{equation*}
\begin{split}
&N(s)\\
&\leq c
\left(\frac{1}{s^2}\int_{B(s)}a_1(x)\,{\rm d}x+\frac{1}{(\log s)^{\beta}}\int_{B(s)}a_2(x)\,{\rm d}x
+s^d\int_{|y|\geq s}\frac{a_2(y)\log(|y|/s)}{|y|^d\log(2+|y|)^{\beta+2}}\right).
\end{split}
\end{equation*}

If there exist $c_1>0$ and $c_2>0$ such that 
for any $x\in {\mathbb R}^d$,
$$a_1(x)\leq c_1(1+|x|)^{2-d}\log(2+|x|), \quad a_2(x)\leq c_2\frac{(\log (2+|x|))^{\beta}}{(1+|x|)^d},$$
then there exists $c_3>0$ such that 
$N(s)\leq c_3\log(2+s)$ for any $s\geq 2$. 
Since $\int_2^{\infty}(sN(s))^{-1}\,{\rm d}s=\infty$, $({\cal E},{\cal F})$ is recurrent.
\end{enumerate}
\end{exam}

\begin{exam}\label{exam:j-d}\rm 
Let $\{a_{ij}(x)\}_{1\leq i,j\leq d}$ be a family of locally integrable Borel measurable functions on ${\mathbb R}^d$  
such that the matrix $(a_{ij}(x))_{1\leq i,j\leq d}$ is symmetric and nonnegative definite for each $x\in {\mathbb R}^d$. 
Assume that $\{a_{ij}(x)\}_{1\leq i,j\leq d}$ is locally uniformly elliptic 
and $J(x,{\rm d}y)$ satisfies Assumption \ref{assum:regular}. 
Under this assumption, if we define the quadratic form $({\cal E},C_0^{\infty}({\mathbb R}^d))$ by
$${\cal E}(u,u)
=\int_{{\mathbb R}^d}\sum_{i,j=1}^d a_{ij}(x)\frac{\partial u}{\partial x_i}(x)\frac{\partial u}{\partial x_j}(x)\,{\rm d}x
+\iint_{{\mathbb R}^d\times{\mathbb R}^d}(u(x)-u(y))^2\,J(x, {\rm d}y){\rm d}x, 
$$
then it is closable on $L^2({\mathbb R}^d)$ and 
its closure $({\cal E},{\cal F})$ is a regular Dirichlet form on $L^2({\mathbb R}^d)$.

We now assume that 
$$\sum_{i,j=1}^d a_{ij}(x)\xi_i\xi_j\leq c(1+|x|)^{2-d}\log(2+|x|)|\xi|^2, \quad x,\xi\in {\mathbb R}^d$$
for some $c>0$ and $J(x,{\rm d}y)$ satisfies the same condition as in Example \ref{exam:rec} (1). 
Then $({\cal E},{\cal F})$ is recurrent by Theorem \ref{thm:test-r}. 
\end{exam}

\subsection*{Acknowledgments}
The author would like to thank Yoshihiko Matsumoto for valuable discussions 
about the proof of Lemma {\ref{lem:derivative}} (3).
He is also grateful to Jian Wang for valuable comments on the draft,  
which in particular lead to improvements of Subsection \ref{subsect:appl}. 

\end{document}